\documentclass[11pt,a4paper]{article}
\usepackage{amsfonts}

\usepackage{graphics}
\usepackage{indentfirst}
\usepackage{latexsym}
\usepackage{amsmath}
\usepackage{amssymb}
\usepackage[dvips]{epsfig}
\usepackage{amscd}
\usepackage{amsthm}
\usepackage{hyperref}
\hypersetup{
	colorlinks=true,
	linkcolor=blue,
	anchorcolor=blue,
	citecolor=blue}
\allowdisplaybreaks[4]
\usepackage{cite}

\newtheorem{theorem}{Theorem}[section]
\newtheorem{remark}{Remark}[section]

\newtheorem{lemma}[theorem]{Lemma}

\newcommand{\mr}{\mathbb{R}}

\newcommand{\bl}{\begin{lemma}}
	\newcommand{\el}{\end{lemma}}
\newcommand{\et}{\end{theorem}}

\newcommand{\la}{\label}

\newcommand{\ka}{\kappa}

\newcommand{\Om}{\Omega}

\newcommand{\bn}{\begin{eqnarray}}
\newcommand{\en}{\end{eqnarray}}
\newcommand{\bnn}{\begin{eqnarray*}}
\newcommand{\enn}{\end{eqnarray*}}

\newcommand{\bnnn}{\begin{eqnarray*}}
\newcommand{\ennn}{\end{eqnarray*}}

\newcommand{\ba}{\begin{aligned}}
\newcommand{\ea}{\end{aligned}}
\newcommand{\be}{\begin{equation}}
\newcommand{\ee}{\end{equation}}

\def\norm[#1]#2{\|#2\|_{#1}}

\def\la{\label}

\makeatletter      
\@addtoreset{equation}{section}
\makeatother       
\title{ Initial and
initial boundary value problems for the Navier-Stokes equations with temperature-dependent transport coefficients and large data
}
\date{ }\author{Hongyu Wang$^a$, Rong Zhang$^b$ \thanks{
		Email addresses: {hyuwang\_a@163.com} (H. Y. Wang), rzhang0921@gmail.com (R.
		Zhang).} \\[3mm]  a. School of Mathematics and Computer Science,\\ Nanchang University,  Nanchang 330031, P. R. China; \\
	b. School of Mathematics and Computer Science, \\Nanchang University,  Nanchang 330031, P. R. China.}

\begin{document}
\maketitle

\begin{abstract}
We consider initial and initial boundary value problems for the
one-dimensional compressible Navier--Stokes equations of an ideal polytropic
gas in unbounded domains. The viscosity and heat conductivity are proportional
to $\theta^\alpha$ and $\theta^\beta$, respectively. For every fixed
$\beta\geq0$ and sufficiently small $\alpha\geq0$, depending on the initial
data, we prove the global  existence and asymptotic stability of a strong solution for arbitrarily large
$H^1$ initial data. Dividing the momentum
equation by the viscosity allows us to estimate the specific-volume derivative
without an initial $H^2$ assumption. 
\end{abstract}

\section{Introduction}\la{1}
The compressible Navier-Stokes system describing the one-dimensional motion of a
viscous heat-conducting perfect polytropic gas can be written in Lagrangian coordinates
in the following form 
\begin{equation}\la{11}
	v_t  =u_x,
\end{equation}
\begin{equation}
	u_t+P_x  =\left(\mu\frac{u_x}{v}\right)_x, \la{12}
\end{equation}
\begin{equation}\la{13}
	\left(e+\frac{u^2}{2}\right)_t+(P u)_x  =\left(\kappa \frac{\theta_x}{v}+\mu \frac{u u_x}{v}\right)_x,
\end{equation}
where $t>0$ is time, $x\in\Omega$ denotes the Lagrangian mass coordinate, and the unknown functions $v>0, u, \theta>0, e>0,$ and $P$ are, respectively, the   specific volume of the gas, fluid velocity, absolute temperature,  internal energy, and pressure; $\mu$ and $\kappa$ are the viscosity and heat conductivity coefficients.
For an ideal polytropic gas, the constitutive relations read
\begin{equation}\la{14}
	P=R \theta / v, \quad e  =c_v \theta+\text { const.},
\end{equation}
where  $R>0$ is the gas constant, and $c_{v}$ is heat capacity at constant volume. We assume that $c_v$ is a positive constant and that the transport coefficients have the form
\begin{equation}
	\mu=\tilde{\mu}\theta^\alpha, \quad \kappa=\tilde{\kappa}\theta^\beta,
	\la{transport}
\end{equation}
where $\tilde{\mu},\tilde{\kappa}>0,$ and       $\alpha,  \beta \geq 0.$

We prescribe the initial conditions
\begin{equation}
	(v(x, 0), u(x, 0), \theta(x, 0))=\left(v_{0}(x), u_{0}(x), \theta_{0}(x)\right), \quad x \in \Omega, \label{15}
\end{equation}
and consider the following three problems:
\begin{enumerate}
	\item[(1)]  Cauchy problem
	\begin{equation}
		\Om=\mr,\quad\lim _{|x| \rightarrow\infty}(v(x,t), u(x,t), \theta(x,t))=(1,0,1), \quad t>0;\label{16}
	\end{equation}
	\item[(2)]  boundary and far-field conditions for $\Omega=(0, \infty)$,
	\be
	u(0, t)=0, \theta_x(0, t)=0, \lim _{x \rightarrow \infty}(v(x, t), u(x, t), \theta(x, t))=(1,0,1), \quad t>0;\la{160}
	\ee
	\item[(3)]  boundary and far-field conditions for $\Omega=(0, \infty)$,
	\be\la{1600}
	u(0, t)=0, \theta(0, t)=1, \lim _{x \rightarrow \infty}(v(x, t), u(x, t), \theta(x, t))=(1,0,1), \quad t>0.\ee
\end{enumerate}

When viscosity $\mu$ and heat conductivity $\kappa$ are positive constants ($\alpha$ = $\beta$ = 0), Kazhikhov and Shelukhin \cite{ka1977} obtained the global existence and uniqueness of strong solutions on a bounded domain with arbitrarily large initial data in $H^1$, which was extended to unbounded domains by Kawashima-Nishida in \cite{kaw1981}. The crucial analysis in \cite{ka1977} was to derive the representation of specific volume $v$, based on which the pointwise upper and lower bounds of density could be obtained. 
Further results on initial boundary value problems can be found in \cite{zlo1989,ant1990} and the references therein. As for the large time behavior, the system has been studied under some smallness conditions on the initial data; see [\citen{hoff1992}, \citen{jiang1998}, \citen{kan1979}] and the references therein. Li--Liang \cite{li2016} subsequently proved that the global solution is asymptotically stable as time tends to infinity  without any smallness assumptions on the initial data.

Within the framework of kinetic theory, the Chapman–Enskog expansion at the first-order approximation reveals that both the viscosity $\mu$ and thermal conductivity $\kappa$ depend solely on temperature ([\citen{cer1994},\citen{cha1990}]). More precisely, in the case of an inverse-power-law intermolecular potential proportional to 
$r^{-a}$, where $r$ denotes the intermolecular distance, these transport coefficients scale with temperature as $\theta^{\frac{a+4}{2a}}$, that is, $\alpha = \beta = \frac{a+4}{2a}$. This gives a square-root dependence ($\theta^{\frac{1}{2}}$) for elastic spheres and  linear dependence for Maxwellian molecules. 

In the case that  the viscosity $\mu$ is a constant and the heat-conductivity $\kappa$
varies with the temperature 
($\alpha = 0$ and $\beta >0$), Jenssen-Karper \cite{jenssen2010}  proved the global existence of a weak solution if $\beta \in (0,\frac{2}{3})$ on a bounded domain. Later on, Pan-Zhang \cite{pan2015} derived the global strong solution, provided that the initial data satisfy
\begin{equation}
	\left(v_0, u_0, \theta_0\right) \in H^1 \times H^2 \times H^2, 
\end{equation}
with $\beta \in (0,\infty)$. Huang--Shi \cite{huang2019} extended this result to initial data in $H^1\times H^1\times H^1$; see also \cite{lpp,wang2016} and the references therein. For unbounded domains, Li--Xu \cite{qian2024} established the large-time behavior of solutions to the initial and
initial boundary value problems with large initial data in unbounded domains, which can be regarded as a natural generalization of Li-Liang’s
result \cite{li2016}. 

When the viscosity $\mu$ is  a function of $v$ and the heat-conductivity $\ka$ varies with both $v$ and $\theta$, it follows from \eqref{11} and \eqref{12} that  \begin{equation}
	\label{volume-viscosity}
	\left(\frac{\mu(v)v_x}{v}\right)_t=u_t+P_x.
\end{equation}
 The works \cite{tan2013,kaw1985,jiang1998,chen2000,daf1982,daf19821} investigated the global well-posedness theory of \eqref{11}-\eqref{13} using Kanel's method \cite{kanel}. When $\mu$ depends on $\theta$, the corresponding identity becomes
\begin{equation}
	\label{temperature-viscosity}
	\left(\frac{\mu(\theta)v_x}{v}\right)_t = u_t +P_x +\frac{\mu_\theta(\theta)}{v}\left(\theta_t v_x-u_x\theta_x\right).
\end{equation}
The last term in \eqref{temperature-viscosity} contains both $\theta_t v_x$ and $u_x\theta_x$. Liu--Yang--Zhao--Zou \cite{liu2014} (see also \cite{wan2017}) proved a Nishida--Smoller type global existence result,  provided that the adiabatic exponent $\gamma$ is close enough to 1 and the oscillation of the temperature is small. Wang--Zhao \cite{wang2016} extended the analysis to the case that the viscosity $\mu$ depends on $v$ and $\theta$, specifically examining models where $\mu$ and $\ka$ take the form $\mu=\tilde{\mu}h(v)\theta^{\alpha}$, $\kappa = \tilde{\kappa}h(v)\theta^\alpha$, in which $\tilde{\mu}$,$\tilde{\ka}$ are positive constants. The authors proved  that under the smallness  assumption on $\alpha\geq0$ the Cauchy
problem of \eqref{11}-\eqref{13} has a global unique smooth solution  provided that the initial data belong to $H^3\times H^3 \times H^3$ and $v^{l_1}+v^{-l_2}\leq Ch(v)$,  for some positive numbers $l_1, l_2\geq1$ and $C>0$. When $l_1, l_2=0$, Sun-Zhang-Zhao \cite{sun2021} established that under the smallness assumption on $\alpha\geq0$ the Dirichlet problem of \eqref{11}-\eqref{13}
admits global strong solutions with initial data in $H^2\times H^2\times H^2$.
 Recently, for the Cauchy problem, Dong and Guo in \cite{dong2025} obtained the existence and asymptotic stability of a global strong solution, provided that $\alpha\geq0$ is suitably small and the initial data satisfy
\begin{equation}
	\left(v_0 -1, u_0, \theta_0 -1\right) \in H^2 \times H^1 \times H^1. 
\end{equation} 

As noted in \cite{dong2025}, that argument requires $v_0-1\in H^2$ and does not extend directly to $H^1$ initial data. The aim of this paper is to study the global existence and the large-time behavior of
solutions to the problem \eqref{11}-\eqref{13} and obtain uniform estimates at this lower regularity.

Throughout this paper, $R,c_v,\tilde\mu,\tilde\kappa$ are fixed positive
constants. Our main result is as follows.

\begin{theorem}\label{lm11}
Let $\beta\geq0$, and suppose that
\begin{equation}\label{190}
 (v_0-1,u_0,\theta_0-1)\in H^1(\Omega),\qquad
 \inf_{\Omega}v_0>0,\qquad \inf_{\Omega}\theta_0>0.
\end{equation}
There exists $\varepsilon_0>0$, depending only on
$R,c_v,\tilde\mu,\tilde\kappa,\beta$, the two initial positive infima
and $\|(v_0-1,u_0,\theta_0-1)\|_{H^1}$, such that, for
$0\leq\alpha\leq\varepsilon_0$, each of the problems
\eqref{11}--\eqref{15}, \eqref{16};
\eqref{11}--\eqref{15}, \eqref{160}; and
\eqref{11}--\eqref{15}, \eqref{1600} has a unique global strong solution
satisfying
\begin{equation}\label{1110}
 \begin{gathered}
 (v-1,u,\theta-1)\in C([0,\infty);H^1)
                    \cap L^\infty(0,\infty;H^1),\\
 v_t\in L^\infty(0,\infty;L^2)\cap L^2(0,\infty;H^1),\\
 (v_x,u_x,\theta_x,u_t,\theta_t,v_{xt},u_{xx},\theta_{xx})
                    \in L^2(\Omega\times(0,\infty)).
 \end{gathered}
\end{equation}
Moreover, for some $C_0>1$ with the same dependence,
\begin{equation}\label{18}
 C_0^{-1}\leq v(x,t)\leq C_0,\qquad
 C_0^{-1}\leq\theta(x,t)\leq C_0
 \quad ((x,t)\in\bar\Omega\times[0,\infty)),
\end{equation}
and, for every $2<p\leq\infty$,
\begin{equation}\label{110}
 \lim_{t\to\infty}\left\{
 \|(v-1,u,\theta-1)(t)\|_{L^p}
 +\|(v_x,u_x,\theta_x)(t)\|_{L^2}\right\}=0.
\end{equation}
\end{theorem}

\begin{remark}
For sufficiently small $\alpha\geq0$, Theorem~\ref{lm11} lowers the
regularity assumption on $v_0-1$ in \cite{dong2025} from $H^2$ to $H^1$.
No smallness is imposed on the initial perturbation, while the admissible
range of $\alpha$ depends on its $H^1$ norm and the initial positive infima.
The case of arbitrary $\alpha\geq0$ remains open in this argument.
\end{remark}

To continue the local solution of Lemma~\ref{lm31}, we need
time-independent positive lower and upper bounds for $v$ and $\theta$,
together with their first derivatives. Such pointwise bounds also underlie
the results in \cite{huang2019,li2016,ka1982,ka1977,nishida1986}.
The estimates in \cite{huang2019,sun2021,qian2024} use either a bounded
domain or constant viscosity. Here the derivatives of $\mu(\theta)$
produce additional terms, which must be controlled without an initial
$H^2$ bound for $v_0-1$.

Under the temporary bounds \eqref{20}, the factors $\mu_x$ and $\mu_t$
carry the small parameter $\alpha$. The representation of $v$, together
with the entropy estimate and the temperature oscillation on unit
intervals, gives the bounds for $v$. A comparison argument on fixed
time intervals then yields a positive lower bound for $\theta$.
For the specific-volume derivative, the identity used in
\cite{dong2025,sun2021} is
\begin{equation}\la{41}
 \left(\frac{\mu}{v}v_x\right)_t
 =u_t+P_x+\frac{\mu'}v(\theta_t v_x-\theta_xu_x).
\end{equation}
The term $\theta_t v_x$ leads to the dependence on
$\|v_0-1\|_{H^2}$ in those estimates. Instead, dividing \eqref{12} by
$\mu$ gives
\begin{equation}\label{2671}
 \left(\frac u\mu\right)_t+\frac{\mu' u\theta_t}{\mu^2}
       +\frac{P_x}\mu
 =\frac{\mu_x}\mu(\ln v)_t+(\ln v)_{xt}.
\end{equation}
Multiplication by $(\ln v)_x$ produces a cross term
$\int_\Omega uv_x/(\mu v)\,dx$ and replaces $\theta_t v_x$ by
$u\theta_t$ in the momentum identity. The resulting products are
estimated in Lemma~\ref{lm25} using only the temporary $H^1$ bounds
and the smallness of $\alpha$.

The temperature estimates are divided into $0\leq\beta\leq1$ and
$\beta>1$. In each case, the multiplier $\theta^\beta\theta_t$
controls $\theta^\beta\theta_x$, and the unit-interval estimate for
$\theta^{\beta+3/2}$ closes the temperature upper bound. The remaining
second derivatives follow from the equations. Section~\ref{2} establishes
these a priori estimates, and Section~\ref{3} proves
Theorem~\ref{lm11} by continuation and the large-time estimates.
\section{A priori estimates}\label{2}
\setcounter{equation}{0}
For $T<\infty$, let $X([0,T])$ be the class of positive solutions of
\eqref{11}--\eqref{15} with the prescribed boundary conditions, initial
data attained in $L^2$, and
\begin{equation}\label{strong-class}
 \begin{gathered}
 (v-1,u,\theta-1)\in L^\infty(0,T;H^1),\\
 v_t\in L^\infty(0,T;L^2)\cap L^2(0,T;H^1),\\
 (v_x,u_x,\theta_x,u_t,\theta_t,v_{xt},u_{xx},\theta_{xx})
                         \in L^2(\Omega\times(0,T)).
 \end{gathered}
\end{equation}
For $M\geq1$, set, as in the continuation argument,
\begin{equation}\label{20}
 \begin{aligned}
 X_M([0,T])=\bigg\{&(v,u,\theta)\in X([0,T]):
 M^{-1}\leq v,\theta\leq M,\\
 &\sup_{0\leq t\leq T}\|(v-1,u,\theta-1)(t)\|_{H^1}^2\\
 &\quad+\int_0^T\|(v_x,u_x,\theta_x,\theta_t,u_{xx},\theta_{xx})\|_{L^2}^2dt
 \leq M\bigg\}.
 \end{aligned}
\end{equation}
In this section, $(v,u,\theta)\in X_M([0,T])$ and $0\leq\alpha\leq1$. The constants $C$ depend only on the fixed coefficients, $\beta$ and the initial quantities in Theorem~\ref{lm11}, and are independent of $M,T,\alpha$. Constants written as $C(M)$ or $C(T,M)$ may additionally depend on the displayed quantities and are used only to justify the calculations at finite time. The further restrictions on $\alpha$ are collected in \eqref{all-smallness}.

The integrations below are first made for smooth approximations and then
passed to the class \eqref{strong-class}. The time-integrated identities,
including the initial values, are justified at the end of
Lemma~\ref{lm280} using only the finite bounds in \eqref{20}.
\subsection*{2.1. Lower order estimates}
We first give the representation of $v$.
\begin{lemma}\label{lm21}
For $x\in[k,k+1]$ and $0\leq s\leq t\leq T$,
\begin{equation}\label{v}
 v(x,t)=\frac{B_k(x,t)Y_k(t)}{B_k(x,s)Y_k(s)}v(x,s)
 +R\int_s^t\frac{B_k(x,t)Y_k(t)}{B_k(x,r)Y_k(r)}
                      \frac{\theta(x,r)}{\mu(x,r)}dr,
\end{equation}
where $k\in\mathbb Z$ on the line and $k\in\mathbb N\cup\{0\}$ on
the half-line, and
\begin{align}
 B_k(x,t)&=v_0(x)\exp\left\{\int_k^x
 \left(\frac u\mu(y,t)-\frac{u_0}{\mu_0}(y)
                          +\int_0^t h(y,r)dr\right)dy\right\},\label{b}\\
 Y_k(t)&=\exp\left\{\int_0^t\sigma(k,r)dr\right\},\label{y}\\
 h&=-u(\mu^{-1})_t-P(\mu^{-1})_x-\frac{\mu_xu_x}{\mu v}
 =\alpha\left(\frac{u\theta_t}{\mu\theta}
                +\frac{R\theta_x}{\mu v}
                -\frac{\theta_xu_x}{\theta v}\right),\label{h}\\
 \sigma&=\frac{u_x}v-\frac{R\theta}{\mu v}.\label{si}
\end{align}
\end{lemma}
\noindent{\it Proof.}
First, rewriting \eqref{12} and integrating over $[k,x]\times[s,t]$, we have
\begin{align}
 &\left(\frac u\mu\right)_t+h=\sigma_x,\notag\\
 &\int_k^x\left(\frac u\mu(y,t)-\frac u\mu(y,s)
                      +\int_s^t h(y,r)dr\right)dy\notag\\
 &\qquad=\ln\frac{v(x,t)}{v(x,s)}
   -R\int_s^t\frac{\theta}{\mu v}(x,r)dr
   -\int_s^t\sigma(k,r)dr,\notag\\
 &v(x,t)=\frac{B_k(x,t)Y_k(t)}{B_k(x,s)Y_k(s)}v(x,s)
       \exp\left\{R\int_s^t\frac\theta{\mu v}(x,r)dr\right\},\notag\\
 &\frac d{dt}\exp\left\{R\int_s^t\frac\theta{\mu v}(x,r)dr\right\}
 =\frac{R\theta(x,t)B_k(x,s)Y_k(s)}
          {\mu(x,t)v(x,s)B_k(x,t)Y_k(t)}.
\end{align}
Integrating the last equality in time and substituting it into the preceding formula gives \eqref{v}.
\qed

The three energy multipliers give the following estimate.
\begin{lemma}\label{lm22}
There holds
\begin{equation}\label{e0}
 \sup_{0\leq t\leq T}\int_\Omega
 \left\{\frac{u^2}2+R(v-\ln v-1)+c_v(\theta-\ln\theta-1)\right\}dx
 +\int_0^T V(t)dt\leq E_0,
\end{equation}
where
\begin{align}
 V(t)&=\int_\Omega\left\{\frac{\mu u_x^2}{v\theta}
              +\frac{\tilde\kappa\theta^\beta\theta_x^2}{v\theta^2}\right\}dx,
 \label{V}\\
 E_0&=2\int_\Omega\left\{\frac{u_0^2}2+R(v_0-\ln v_0-1)
                       +c_v(\theta_0-\ln\theta_0-1)\right\}dx.
 \label{E0}
\end{align}
\end{lemma}
\noindent{\it Proof.}
Subtracting the kinetic energy identity from \eqref{13} gives
\begin{align}
 \left(\frac{u^2}2\right)_t+(Pu)_x
 &=Pu_x+\left(\frac{\mu uu_x}v\right)_x-\frac{\mu u_x^2}v,
 \notag\\
 c_v\theta_t+\frac{R\theta u_x}v
 &=\left(\frac{\tilde\kappa\theta^\beta\theta_x}v\right)_x
                                     +\frac{\mu u_x^2}v.
 \label{13a}
\end{align}
Multiplying \eqref{11}, \eqref{12} and \eqref{13a} by
$R(1-v^{-1})$, $u$ and $1-\theta^{-1}$, respectively, we obtain
\begin{align}
 &\left\{\frac{u^2}2+R(v-\ln v-1)+c_v(\theta-\ln\theta-1)\right\}_t
 +\frac{\mu u_x^2}{v\theta}
 +\frac{\tilde\kappa\theta^{\beta-2}\theta_x^2}v\notag\\
 &\qquad=\left\{\frac{\mu uu_x}v
 +\frac{\tilde\kappa\theta^\beta(1-\theta^{-1})\theta_x}v
 -R\left(\frac\theta v-1\right)u\right\}_x.
 \label{entropy-identity}
\end{align}
Integrating the resulting
equality over $\Om\times (0,T)$ yields \eqref{e0}.
\qed

Then, we use Lemmas~\ref{lm21} and \ref{lm22} to estimate $v$.
\begin{lemma}\label{lm23}
There are positive constants $C_1,C_2,C_3$, independent of $M,T,\alpha$,
such that
\begin{equation}\label{volume-smallness}
 0\leq\alpha\leq M^{-1},\qquad
 \alpha C_1(1+M)^3\leq\min\{1,C_2/4\}
\end{equation}
imply
\begin{equation}\label{vvv}
 C_3^{-1}\leq v(x,t)\leq C_3
 \quad ((x,t)\in\bar\Omega\times[0,T]).
\end{equation}
\end{lemma}
\noindent{\it Proof.}
If $E_0=0$, \eqref{entropy-identity} gives $(v,u,\theta)=(1,0,1)$.
Suppose $E_0>0$. Jensen's inequality gives
\begin{align}
 &R\left\{\int_k^{k+1}v\,dx
           -\ln\left(\int_k^{k+1}v\,dx\right)-1\right\}\notag\\
 &\quad+c_v\left\{\int_k^{k+1}\theta\,dx
           -\ln\left(\int_k^{k+1}\theta\,dx\right)-1\right\}
 \leq\frac{E_0}2,\notag\\
 &\alpha_1\leq\int_k^{k+1}v\,dx\leq\alpha_2,
 \qquad
 \alpha_1\leq\int_k^{k+1}\theta\,dx\leq\alpha_2,
 \label{v0}
\end{align}
where $0<\alpha_1<1<\alpha_2$ are the two roots of
$y-\ln y-1=E_0/(2\min\{R,c_v\})$. In particular, there are
$a_k(t),b_k(t)\in[k,k+1]$ for which
\begin{equation}\label{v1}
 \alpha_1\leq v(a_k(t),t)\leq\alpha_2,
 \qquad \alpha_1\leq\theta(b_k(t),t)\leq\alpha_2.
\end{equation}
Since $\alpha\leq1/M$,
\begin{equation}\label{ga}
 \tilde\mu/e\leq\tilde\mu M^{-1/M}
 \leq\mu\leq\tilde\mu M^{1/M}\leq e\tilde\mu.
\end{equation}
From \eqref{20}, \eqref{e0} and \eqref{h},
\begin{align}
 \left|\int_k^x\left(\frac u\mu(y,t)-\frac u\mu(y,s)\right)dy\right|
 &\leq\frac e{\tilde\mu}(\|u(t)\|_{L^2}+\|u(s)\|_{L^2})\leq C,
 \notag\\
 \int_s^t\int_k^{k+1}|h|\,dxdr
 &\leq\alpha\left\{\frac{eM}{\tilde\mu}
          \sup_r\|u(r)\|_{L^2}\int_s^t\|\theta_t\|_{L^2}dr
     +\frac{eRM}{\tilde\mu}\int_s^t\|\theta_x\|_{L^2}dr\right.\notag\\
 &\hspace{42mm}\left.+M^2\int_s^t\|u_x\|_{L^2}\|\theta_x\|_{L^2}dr\right\}
 \notag\\
 &\leq\alpha\left\{\frac e{\tilde\mu}
                 (M^2+RM^{3/2})\sqrt{t-s}+M^3\right\}
 \notag\\
 &\leq\alpha C_1(1+M)^3(1+t-s),\label{h-bound}\\
 C^{-1}e^{-\alpha C_1(1+M)^3(t-s)}
 &\leq\frac{B_k(x,t)}{B_k(x,s)}
       \leq Ce^{\alpha C_1(1+M)^3(t-s)}.
 \label{B-ratio}
\end{align}

Whenever $\min_{[k,k+1]}\theta<\alpha_1/2$, \eqref{v1} implies
\begin{align}
 \int_{\alpha_1/2}^{\alpha_1}z^{\beta/2-1}dz
 &\leq\int_k^{k+1}\theta^{\beta/2-1}|\theta_x|dx
 \leq\left(\frac{\alpha_2}{\tilde\kappa}V(t)\right)^{1/2},
 \notag\\
 V(t)&\geq
 \frac{\tilde\kappa}{\alpha_2}
       \left(\int_{\alpha_1/2}^{\alpha_1}z^{\beta/2-1}dz\right)^2>0,
 \label{bad-threshold}\\
 \int_s^t\left(\frac{\alpha_1}2
              -\min_{[k,k+1]}\theta(\cdot,r)\right)_+dr
 &\leq\frac{\alpha_1\alpha_2}{2\tilde\kappa}
       \left(\int_{\alpha_1/2}^{\alpha_1}z^{\beta/2-1}dz\right)^{-2}
                            \int_s^tV(r)dr,\notag\\
 \int_s^t\min_{[k,k+1]}\theta(\cdot,r)dr
 &\geq\frac{\alpha_1}2\left[t-s-\frac{E_0\alpha_2}{2\tilde\kappa}
       \left(\int_{\alpha_1/2}^{\alpha_1}z^{\beta/2-1}dz\right)^{-2}\right].
 \label{cell-window}
\end{align}
For $\beta=0$, the integral in \eqref{bad-threshold} is $\ln2$.
Integrating $(u/\mu)_t+h=\sigma_x$ over a unit interval gives
\begin{align}
 &\int_s^t\sigma(k,r)dr
 =\int_s^t\int_k^{k+1}\left(\frac{u_x}v
                         -\frac{R\theta}{\mu v}\right)dxdr\notag\\
 &\quad-\int_k^{k+1}\int_k^x
             \left(\frac u\mu(y,t)-\frac u\mu(y,s)\right)dydx
       -\int_s^t\int_k^{k+1}\int_k^x h(y,r)dydxdr\notag\\
 &\quad\leq\frac1{2R}\int_s^t\int_k^{k+1}\frac{\mu u_x^2}{v\theta}dxdr
       -\frac R2\int_s^t\int_k^{k+1}\frac\theta{\mu v}dxdr
       +C+\alpha C_1(1+M)^3(t-s)\notag\\
 &\quad\leq C+\alpha C_1(1+M)^3(t-s)
       -\frac R{2e\tilde\mu\alpha_2}
                       \int_s^t\min_{[k,k+1]}\theta(\cdot,r)dr\notag\\
 &\quad\leq C-\bigl(C_2-\alpha C_1(1+M)^3\bigr)(t-s),
 \label{stress-average}
\end{align}
where $C_2=R\alpha_1/(4e\tilde\mu\alpha_2)$. Consequently,
\begin{align}
 0<\frac{B_k(x,t)Y_k(t)}{B_k(x,s)Y_k(s)}
 &\leq C e^{-C_2(t-s)/2},\label{kernel-upper}\\
 v(x,t)&\leq Ce^{-C_2t/2}
                +C\int_0^t e^{-C_2(t-r)/2}\theta(x,r)dr.
 \label{v-upper-integral}
\end{align}
Also,
\begin{align}
 &\left|\theta^{(\beta+1)/2}(x,t)
          -\theta^{(\beta+1)/2}(b_k(t),t)\right|\notag\\
 &\qquad\leq\frac{\beta+1}2\int_k^{k+1}
                            \theta^{(\beta-1)/2}|\theta_x|dx\notag\\
 &\qquad\leq C V(t)^{1/2}
          \left(\int_k^{k+1}\theta v\,dx\right)^{1/2}
 \leq C\left(V(t)\max_{[k,k+1]}v\right)^{1/2},\notag\\
 &\theta(x,t)\leq C+C V(t)\max_{[k,k+1]}v,\notag\\
 &\max_{[k,k+1]}v(\cdot,t)
 \leq C+C\int_0^t V(r)\max_{[k,k+1]}v(\cdot,r)dr\notag\\
 &\qquad
 \leq C\exp\left(C\int_0^t V(r)dr\right)\leq C.
 \label{v-upper}
\end{align}

Choose $T_0$ such that
\begin{equation*}
 T_0>1+\frac{2E_0\alpha_2}{\tilde\kappa}
       \left(\int_{\alpha_1/2}^{\alpha_1}z^{\beta/2-1}dz\right)^{-2}.
\end{equation*}
For $s\leq r\leq\min\{s+T_0,T\}$,
\eqref{B-ratio} gives
$C^{-1}\leq B_k(x,r)/B_k(x,s)\leq C$, with $C$ independent of $s$.
Dividing \eqref{v} by its kernel and integrating in $x$, we obtain
\begin{align}
 &\frac{Y_k(s)}{Y_k(r)}\int_k^{k+1}
                  \frac{B_k(x,s)}{B_k(x,r)}v(x,r)dx\notag\\
 &\quad=\int_k^{k+1}v(x,s)dx
       +R\int_s^r\frac{Y_k(s)}{Y_k(\tau)}
           \int_k^{k+1}\frac{B_k(x,s)}{B_k(x,\tau)}
                          \frac{\theta(x,\tau)}{\mu(x,\tau)}dx\,d\tau,
 \notag\\
 &\frac{Y_k(s)}{Y_k(r)}
 \leq C+C\int_s^r\frac{Y_k(s)}{Y_k(\tau)}d\tau
 \leq Ce^{C(r-s)}\leq C,\notag\\
 &\frac{B_k(x,r)Y_k(r)}{B_k(x,s)Y_k(s)}\geq C^{-1}
 \qquad(0\leq r-s\leq T_0).\label{kernel-lower}
\end{align}
Thus
\begin{align}
 v(x,t)&\geq C^{-1}\inf_\Omega v_0,
                   \qquad 0\leq t\leq\min\{T_0,T\},\notag\\
 v(x,t)&\geq C^{-1}\int_{t-T_0}^t
                      \min_{[k,k+1]}\theta(\cdot,r)dr,
                   \qquad T_0\leq t\leq T,\notag\\
 &\geq\frac{\alpha_1}{2C}
              \left[T_0-\frac{E_0\alpha_2}{2\tilde\kappa}
       \left(\int_{\alpha_1/2}^{\alpha_1}z^{\beta/2-1}dz\right)^{-2}\right]>0.
 \label{v-lower}
\end{align}
The constants are independent of $k$, and \eqref{vvv} follows.

We shall also use the following lower bound:
\begin{equation}\label{theta-lower}
 \theta(x,t)\geq C_4^{-1}\quad
 ((x,t)\in\bar\Omega\times[0,T]),
\end{equation}
where $C_4>1$ depends only on the same quantities.

Completing the square in \eqref{13a}, we have
\begin{align}
 \theta_t-\frac{\tilde\kappa}{c_v}
                   \left(\frac{\theta^\beta\theta_x}v\right)_x
 &=\frac\mu{c_vv}\left(u_x-\frac{R\theta}{2\mu}\right)^2
                  -\frac{R^2\theta^{2-\alpha}}{4c_v\tilde\mu v}
 \geq-C\theta^{2-\alpha},\notag\\
 \theta^{2-\alpha}&\leq\theta
 \qquad(0<\theta\leq1/2,\ 0\leq\alpha\leq1).
 \label{theta-comparison}
\end{align}
For $s\in[0,T]$, let $q$ solve
\begin{equation}
 q'=-Cq,\qquad q(s)=\min\{1/2,\inf_\Omega\theta(\cdot,s)\}.
\end{equation}
Since $0<q\leq1/2$, \eqref{20} gives
\begin{align}
 \int_\Omega\bigl((q-\theta)_++(q-\theta)_+^2\bigr)dx
 &\leq|\{\theta<q\}|
 \leq4\int_\Omega(\theta-1)^2dx\leq4M,\notag\\
 \frac12\frac d{dt}\int_\Omega(q-\theta)_+^2dx
 &=\int_\Omega(q-\theta)_+(q'-\theta_t)dx.
 \label{moving-level-chain}
\end{align}
Multiplying \eqref{theta-comparison} by $-(q-\theta)_+$, we obtain
\begin{align}
 \frac12\frac d{dt}\|(q-\theta)_+\|_{L^2}^2
 &\leq-\frac{\tilde\kappa}{c_v}
             \int_{\{\theta<q\}}\frac{\theta^\beta\theta_x^2}v dx
          +C\int_{\{\theta<q\}}(q-\theta)(\theta-q)dx
 \leq0,\notag\\
 \inf_\Omega\theta(\cdot,t)
 &\geq\min\{1/2,\inf_\Omega\theta(\cdot,s)\}e^{-C(t-s)}.
 \label{theta-forward}
\end{align}
There is no boundary contribution, since $(q-\theta)_+=0$ in
\eqref{1600} and $\theta_x=0$ in \eqref{160}.
For $T_0\leq t\leq T$, the entropy estimate gives
\begin{equation*}
 \frac1{T_0}\int_{t-T_0}^tV(r)dr\leq\frac{E_0}{2T_0}.
\end{equation*}
Choose a time $s\in(t-T_0,t)$ such that
\begin{equation*}
 V(s)\leq\frac{E_0}{T_0}
 <\frac{\tilde\kappa}{\alpha_2}
       \left(\int_{\alpha_1/2}^{\alpha_1}z^{\beta/2-1}dz\right)^2.
\end{equation*}
For every $k$ and $x\in[k,k+1]$, \eqref{v0}--\eqref{v1} imply
\begin{align*}
 \left|\int_{\theta(x,s)}^{\theta(b_k(s),s)}z^{\beta/2-1}dz\right|
 &\leq\int_k^{k+1}\theta^{\beta/2-1}|\theta_x|dx\\
 &\leq\left(\frac{\alpha_2}{\tilde\kappa}V(s)\right)^{1/2}
 <\int_{\alpha_1/2}^{\alpha_1}z^{\beta/2-1}dz.
\end{align*}
Since $\theta(b_k(s),s)\geq\alpha_1$, this gives
$\inf_\Omega\theta(\cdot,s)\geq\alpha_1/2$. The same time $s$
works for all unit intervals.
Applying \eqref{theta-forward} gives
\begin{align}
 \inf_\Omega\theta(\cdot,t)
 &\geq\min\{1/2,\inf_\Omega\theta_0\}e^{-CT_0},
                                  &&0\leq t\leq T_0,\notag\\*
 \inf_\Omega\theta(\cdot,t)
 &\geq\min\{1/2,\alpha_1/2\}e^{-CT_0},
                                  &&T_0\leq t\leq T.
\end{align}
\nopagebreak[4]
This proves \eqref{theta-lower}.\qed

For the temperature estimates, denote
\begin{equation}\label{rt}
 R_T=1+\sup_{\bar\Omega\times[0,T]}\theta(x,t).
\end{equation}
We divide the estimates into two cases: $0\leq\beta\leq1$ and $\beta>1$.

\subsubsection*{2.1.1. Case $0\leq\beta\leq1$}
\begin{lemma}\label{lm24}
Under \eqref{volume-smallness}, if $0\leq\beta\leq1$, then
\begin{align}
 &\sup_{0\leq t\leq T}\int_\Omega
       \bigl((\theta-1)^2+\theta u^2\bigr)dx
 +\int_0^T\int_\Omega(\theta u_x^2+\theta^\beta\theta_x^2)dxdt
 \leq C R_T^{2\beta/(1+\beta)},\label{240}\\
 &\sup_{0\leq t\leq T}\int_\Omega u^4dx
 +\int_0^T\int_\Omega u^2u_x^2dxdt
 +\int_0^T\|(\theta-3/2)_+\|_{L^\infty}^2dt
 \leq C R_T^{\beta(1-\beta)/(1+\beta)}.\label{241}
\end{align}
\end{lemma}
\noindent{\it Proof.}
First, for $a>0$, denote
\begin{equation}\label{243}
 (\theta>a)(t)=\{x\in\Omega:\theta(x,t)>a\},\qquad
 (\theta<a)(t)=\{x\in\Omega:\theta(x,t)<a\}.
\end{equation}
For $a>1$, the entropy estimate gives
\begin{equation}\label{hot-measure}
 \sup_{0\leq t\leq T}\left\{\|u\|_{L^2}^2+\|v-1\|_{L^2}^2
 +\int_{(\theta>a)(t)}\theta\,dx
 +\int_{\{\theta\leq a\}}(\theta-1)^2dx\right\}\leq C(a).
\end{equation}
Multiplying \eqref{13a} by $(\theta-2)_+$ and integrating over
$\Omega\times(0,t)$, we obtain
\begin{align}
 &\frac{c_v}{2}\int_\Omega(\theta-2)_+^2dx
 +\tilde\kappa\int_0^t\int_{(\theta>2)(s)}
                    \frac{\theta^\beta\theta_x^2}{v}dxds\notag\\
 &=\frac{c_v}{2}\int_\Omega(\theta_0-2)_+^2dx
 -R\int_0^t\int_\Omega\frac{\theta u_x}{v}(\theta-2)_+dxds
 +\int_0^t\int_\Omega\frac{\mu u_x^2}{v}(\theta-2)_+dxds.
 \label{246}
\end{align}
To estimate the last term, we multiply \eqref{12} by
$2u(\theta-2)_+$ and integrate to get
\begin{align}
 &\int_\Omega u^2(\theta-2)_+dx
 +2\int_0^t\int_\Omega\frac{\mu u_x^2}{v}(\theta-2)_+dxds\notag\\
 &=\int_\Omega u_0^2(\theta_0-2)_+dx
 +2R\int_0^t\int_\Omega\frac{\theta u_x}{v}(\theta-2)_+dxds\notag\\
 &\quad+2R\int_0^t\int_{(\theta>2)(s)}\frac{\theta u\theta_x}{v}dxds
 -2\int_0^t\int_{(\theta>2)(s)}\frac{\mu uu_x\theta_x}{v}dxds
 +\int_0^t\int_{(\theta>2)(s)}u^2\theta_tdxds.
 \label{247}
\end{align}
Adding \eqref{246} and \eqref{247}, and using \eqref{13a}, we have
\begin{align}
 &\int_\Omega\left\{\frac{c_v}2(\theta-2)_+^2
                                  +u^2(\theta-2)_+\right\}dx
 +\tilde\kappa\int_0^t\int_{(\theta>2)(s)}
                                  \frac{\theta^\beta\theta_x^2}v dxds
 \notag\\*
 &\quad+\int_0^t\int_\Omega\frac{\mu u_x^2}v(\theta-2)_+dxds
 \notag\\*
 &=\int_\Omega\left\{\frac{c_v}2(\theta_0-2)_+^2
                                  +u_0^2(\theta_0-2)_+\right\}dx
 +R\int_0^t\int_\Omega\frac{\theta u_x}v(\theta-2)_+dxds\notag\\
 &\quad+2R\int_0^t\int_{(\theta>2)(s)}
                      \frac{\theta u\theta_x}v dxds
       -2\int_0^t\int_{(\theta>2)(s)}
                      \frac{\mu uu_x\theta_x}v dxds\notag\\
 &\quad+\frac1{c_v}\int_0^t\int_{(\theta>2)(s)}
                    u^2\left(\frac{\mu u_x^2}v-\frac{R\theta u_x}v\right)dxds
       +\frac{\tilde\kappa}{c_v}\int_0^t\int_{(\theta>2)(s)}
                    u^2\left(\frac{\theta^\beta\theta_x}v\right)_x dxds
 \notag\\
 &=\int_\Omega\left\{\frac{c_v}2(\theta_0-2)_+^2
                                  +u_0^2(\theta_0-2)_+\right\}dx
                                      +\sum_{i=1}^5 I_i.
 \label{hot-identity}
\end{align}
By \eqref{ga}, \eqref{vvv} and \eqref{hot-measure},
\begin{align}
 |I_1|
 &\leq\frac18\int_0^t\int_\Omega
                  \frac{\mu u_x^2}v(\theta-2)_+dxds
       +C\int_0^t\int_{(\theta>2)(s)}
                         \theta^2(\theta-2)\,dxds\notag\\
 &\leq\frac18\int_0^t\int_\Omega
                  \frac{\mu u_x^2}v(\theta-2)_+dxds
       +C\int_0^t\|(\theta-3/2)_+\|_{L^\infty}^2ds,\notag\\
 |I_2|+|I_3|
 &\leq\eta\int_0^t\int_{(\theta>2)(s)}
                  \frac{\theta^\beta\theta_x^2}v dxds
       +C(\eta)\int_0^t\int_{(\theta>2)(s)}
                  \bigl(\theta^{2-\beta}u^2
                              +\theta^{-\beta}u^2u_x^2\bigr)dxds\notag\\
 &\leq\eta\int_0^t\int_{(\theta>2)(s)}
                  \frac{\theta^\beta\theta_x^2}v dxds
       +C(\eta)\int_0^t\left\{\|(\theta-3/2)_+\|_{L^\infty}^2
                                    +\int_\Omega u^2u_x^2dx\right\}ds,
 \notag\\
 |I_4|
 &\leq C\int_0^t\int_{(\theta>2)(s)}
                      (u^2u_x^2+\theta^2u^2)dxds\notag\\
 &\leq C\int_0^t\int_\Omega u^2u_x^2dxds
                         +C\int_0^t\|(\theta-3/2)_+\|_{L^\infty}^2ds.
 \label{hot-first-four}
\end{align}
For the last term use
\begin{equation}
 \varphi_\eta(\theta)=
 \begin{cases}
 0,&\theta\leq2,\\
 (\theta-2)/\eta,&2<\theta<2+\eta,\\
 1,&\theta\geq2+\eta.
 \end{cases}
\end{equation}
The bounds in \eqref{20} give
\begin{align}
 &\int_0^T(\|u_x\|_{L^\infty}^2+\|\theta_x\|_{L^\infty}^2)dt\notag\\
 &\quad\leq2\int_0^T
       (\|u_x\|_{L^2}\|u_{xx}\|_{L^2}
          +\|\theta_x\|_{L^2}\|\theta_{xx}\|_{L^2})dt
 \leq2M,\notag\\
 &\int_0^T\int_\Omega
       (u_x^4+\theta_x^4+u_x^2v_x^2+\theta_x^2v_x^2+u_x^2\theta_x^2)dxdt
 \notag\\
 &\quad\leq2\sup_{s\leq T}\int_\Omega(v_x^2+u_x^2+\theta_x^2)dx
       \int_0^T(\|u_x\|_{L^\infty}^2+\|\theta_x\|_{L^\infty}^2)dt
 \leq4M^2.\label{finite-products}
\end{align}
Thus
\begin{align*}
 \left(\frac{\theta^\beta\theta_x}v\right)_x
 &=\frac{\theta^\beta}v\theta_{xx}
   +\frac{\beta\theta^{\beta-1}}v\theta_x^2
   -\frac{\theta^\beta}{v^2}\theta_xv_x.
\end{align*}
Hence
\begin{align*}
 &\int_0^T\int_\Omega
       \left|\left(\frac{\theta^\beta\theta_x}v\right)_x\right|^2dxdt\\*
 &\quad\leq C(M)\int_0^T\int_\Omega
       (\theta_{xx}^2+\theta_x^4+\theta_x^2v_x^2)dxdt\leq C(M),\\
 &\int_0^t\int_\Omega u^2
       \left|\left(\frac{\theta^\beta\theta_x}v\right)_x\right|dxds\\
 &\quad\leq\sup_{s\leq t}\|u(s)\|_{L^\infty}
       \left(\int_0^t\int_\Omega u^2dxds\right)^{1/2}
       \left(\int_0^t\int_\Omega
       \left|\left(\frac{\theta^\beta\theta_x}v\right)_x\right|^2dxds\right)^{1/2}
 \leq C(M)\sqrt t,\\
 &\int_0^t\int_\Omega
       \left|\frac{uu_x\theta^\beta\theta_x}v\right|dxds
 \leq C(M)\int_0^t\int_\Omega(u_x^2+\theta_x^2)dxds\leq C(M).
\end{align*}
Consequently, integration by parts and $\varphi_\eta'\geq0$ give
\begin{align}
 I_5
 &=\lim_{\eta\downarrow0}\frac{\tilde\kappa}{c_v}
       \int_0^t\int_\Omega\varphi_\eta(\theta)u^2
                  \left(\frac{\theta^\beta\theta_x}v\right)_x dxds
 \notag\\
 &=\lim_{\eta\downarrow0}\left\{
   -\frac{\tilde\kappa}{c_v}\int_0^t\int_\Omega
           \varphi_\eta'(\theta)\frac{u^2\theta^\beta\theta_x^2}v dxds
                       \right.\notag\\
 &\hspace{29mm}\left.
   -\frac{2\tilde\kappa}{c_v}\int_0^t\int_\Omega
           \varphi_\eta(\theta)\frac{uu_x\theta^\beta\theta_x}v dxds
                    \right\}\notag\\
 &\leq\lim_{\eta\downarrow0}
   \left[-\frac{2\tilde\kappa}{c_v}\int_0^t\int_\Omega
           \varphi_\eta(\theta)\frac{uu_x\theta^\beta\theta_x}v dxds\right]
 \notag\\
 &=-\frac{2\tilde\kappa}{c_v}\int_0^t\int_{(\theta>2)(s)}
                       \frac{uu_x\theta^\beta\theta_x}v dxds\notag\\
 &\leq\varepsilon\int_0^t\int_{(\theta>2)(s)}
                            \frac{\theta^\beta\theta_x^2}v dxds
       +C(\varepsilon)R_T^\beta\int_0^t\int_\Omega u^2u_x^2dxds.
 \label{hot-interface}
\end{align}
Choose $\eta,\varepsilon>0$ sufficiently small in
\eqref{hot-first-four} and \eqref{hot-interface}. Since
\begin{align}
 \int_\Omega\bigl((\theta-1)^2+\theta u^2\bigr)dx
 &\leq C+C\int_\Omega\bigl((\theta-2)_+^2+u^2(\theta-2)_+\bigr)dx,
 \notag\\
 \int_0^t\int_{\{\theta\leq3\}}
                    (\theta u_x^2+\theta^\beta\theta_x^2)dxds
 &\leq C\int_0^t V(s)ds\leq C,
\end{align}
we have
\begin{align}
 &\sup_{0\leq s\leq t}\int_\Omega
        \bigl((\theta-1)^2+\theta u^2\bigr)dx
       +\int_0^t\int_\Omega(\theta u_x^2+\theta^\beta\theta_x^2)dxds
 \notag\\
 &\qquad\leq C+C\int_0^t\|(\theta-3/2)_+\|_{L^\infty}^2ds
                  +C R_T^\beta\int_0^t\int_\Omega u^2u_x^2dxds.
 \label{hot-bound}
\end{align}

Multiplying \eqref{12} by $u^3$, we obtain
\begin{align}
 \frac14\int_\Omega u^4dx
 +3\int_0^t\int_\Omega\frac{\mu u^2u_x^2}v dxds
 &=\frac14\int_\Omega u_0^4dx
   +3R\int_0^t\int_{\{\theta>2\}}
                       \left(\frac\theta v-1\right)u^2u_xdxds\notag\\
 &\quad+3R\int_0^t\int_{\{\theta\leq2\}}
                       \left(\frac\theta v-1\right)u^2u_xdxds.
 \label{quartic-identity}
\end{align}
The two integrals satisfy
\begin{align}
 &3R\left|\int_0^t\int_{\{\theta>2\}}
                       \left(\frac\theta v-1\right)u^2u_xdxds\right|\notag\\
 &\qquad\leq\int_0^t\int_\Omega\frac{\mu u^2u_x^2}v dxds
                      +C\int_0^t\|(\theta-3/2)_+\|_{L^\infty}^2ds,
 \notag\\
 &3R\left|\int_0^t\int_{\{\theta\leq2\}}
                       \left(\frac\theta v-1\right)u^2u_xdxds\right|\notag\\
 &\qquad\leq C\int_0^t V(s)ds
      +C\int_0^t\|u\|_{L^\infty}^4
            \int_{\{\theta\leq2\}}
               \bigl((\theta-1)^2+(v-1)^2\bigr)dxds\notag\\
 &\qquad\leq C+C\int_0^t\|u\|_{L^\infty}^4ds.
 \label{quartic-pressure}
\end{align}
For any $\zeta>0$, the one-dimensional inequality gives
\begin{align}
 \|u\|_{L^\infty}^4
 &\leq4\int_\Omega |u|^3|u_x|dx\notag\\
 &\leq4\int_{\{\theta\leq2\}}|u|^3|u_x|dx
    +4\int_{\{\theta>2\}}|u|^{3/2}
          (|u|\theta^{1/4})|uu_x|^{1/2}
                              \theta^{-1/4}|u_x|^{1/2}dx\notag\\
 &\leq\zeta\int_{\{\theta\leq2\}}\theta u^6dx
       +\zeta\int_{\{\theta>2\}}
                        (u^6+\theta u^4+u^2u_x^2)dx
       +C(\zeta)\int_\Omega\theta^{-1}u_x^2dx\notag\\
 &\leq C\zeta\|u\|_{L^\infty}^4
                 +C\zeta\int_\Omega\frac{\mu u^2u_x^2}v dx
                 +C(\zeta)V(t),\notag\\
 \|u\|_{L^\infty}^4
 &\leq\delta\int_\Omega\frac{\mu u^2u_x^2}v dx+C(\delta)V(t)
 \qquad(\delta>0).
 \label{quartic-sobolev}
\end{align}
Substitution into \eqref{quartic-identity}, with $\delta$ sufficiently
small, yields
\begin{align}
 &\sup_{0\leq s\leq t}\int_\Omega u^4dx
       +\int_0^t\int_\Omega u^2u_x^2dxds
 \leq C+C\int_0^t\|(\theta-3/2)_+\|_{L^\infty}^2ds,\label{quartic-bound}\\
 &\sup_{0\leq s\leq t}\int_\Omega
                    \bigl((\theta-1)^2+\theta u^2\bigr)dx
       +\int_0^t\int_\Omega(\theta u_x^2+\theta^\beta\theta_x^2)dxds
 \notag\\
 &\qquad\leq C R_T^\beta
                \left(1+\int_0^t\|(\theta-3/2)_+\|_{L^\infty}^2ds\right).
 \label{low-order-closure}
\end{align}
Finally,
\begin{align}
 \|(\theta-3/2)_+\|_{L^\infty}^2
 &\leq2\int_\Omega(\theta-3/2)_+|\theta_x|dx\notag\\
 &\leq\delta\|(\theta-3/2)_+\|_{L^\infty}^2
                         \int_{\{\theta>3/2\}}\theta\,dx
       +C(\delta)\int_{\{\theta>3/2\}}\theta^{-1}\theta_x^2dx\notag\\
 &\leq\frac12\|(\theta-3/2)_+\|_{L^\infty}^2
                         +C\int_\Omega\theta^{-1}\theta_x^2dx.
 \label{hot-peak}
\end{align}
For $\beta=1$, \eqref{e0} and \eqref{hot-peak} give
\begin{equation}
 \int_0^T\|(\theta-3/2)_+\|_{L^\infty}^2dt
 \leq C\int_0^T\int_\Omega\theta^{-1}\theta_x^2dxdt\leq C.
\end{equation}
Together with \eqref{quartic-bound} and \eqref{low-order-closure},
this proves \eqref{240}--\eqref{241} for $\beta=1$.

For $0\leq\beta<1$, weighted Young's inequality gives
\begin{align}
 \theta^{-1}
 &=(\theta^\beta)^{(1-\beta)/2}
             (\theta^{\beta-2})^{(1+\beta)/2}
 \leq\varepsilon\theta^\beta
       +C\varepsilon^{-(1-\beta)/(1+\beta)}\theta^{\beta-2},\notag\\
 \int_0^T\|(\theta-3/2)_+\|_{L^\infty}^2dt
 &\leq\varepsilon\int_0^T\int_\Omega\theta^\beta\theta_x^2dxdt
                         +C\varepsilon^{-(1-\beta)/(1+\beta)}.
 \label{hot-interpolation}
\end{align}
Taking $\varepsilon=(2C R_T^\beta)^{-1}$ in
\eqref{low-order-closure},
\begin{align}
 &\sup_{0\leq t\leq T}\int_\Omega
                    \bigl((\theta-1)^2+\theta u^2\bigr)dx
       +\int_0^T\int_\Omega(\theta u_x^2+\theta^\beta\theta_x^2)dxdt
 \notag\\
 &\qquad\leq\frac12\int_0^T\int_\Omega\theta^\beta\theta_x^2dxdt
                     +C R_T^{2\beta/(1+\beta)}.
\end{align}
Using this estimate in \eqref{hot-interpolation} with
$\varepsilon=R_T^{-\beta}$, we have
\begin{equation}
 \int_0^T\|(\theta-3/2)_+\|_{L^\infty}^2dt
 \leq C R_T^{-\beta+2\beta/(1+\beta)}
       +C R_T^{\beta(1-\beta)/(1+\beta)}
 \leq C R_T^{\beta(1-\beta)/(1+\beta)}.
\end{equation}
Together with \eqref{quartic-bound}, this proves the lemma.
\qed

\begin{lemma}\label{lm25}
In addition to \eqref{volume-smallness}, assume that
\begin{equation}\label{derivative-smallness}
 \begin{gathered}
 \alpha\left(\frac{2eC_3C_4}{\tilde\mu}+\sqrt2 C_3^2C_4\right)
                                      (1+M)^2\leq1,\\
 \alpha^2(e\tilde\mu C_4)^2M\leq1.
 \end{gathered}
\end{equation}
For $0\leq\beta\leq1$,
\begin{align}
 &\sup_{0\leq t\leq T}\|v_x(t)\|_{L^2}^2
       +\int_0^T\int_\Omega\bigl((1+\theta)v_x^2+u_x^2\bigr)dxdt
 \leq
 C R_T^{\beta/(1+\beta)}.\label{266}
\end{align}

\end{lemma}
\noindent{\it Proof.}
First, \eqref{e0}, \eqref{ga}, \eqref{vvv} and Lemma~\ref{lm24} give
\begin{align}
 \int_0^T\int_\Omega u_x^2dxdt
 &\leq\left(\int_0^T\int_\Omega\theta u_x^2dxdt\right)^{1/2}
       \left(\int_0^T\int_\Omega\theta^{-1}u_x^2dxdt\right)^{1/2}
 \notag\\
 &\leq
 C R_T^{\beta/(1+\beta)}.\label{ux-low}\\
 \int_0^T\int_\Omega\theta^{-1}\theta_x^2dxdt
 &\leq
 C R_T^{\beta(1-\beta)/(1+\beta)}.\label{inverse-gradient}
\end{align}
Here \eqref{hot-interpolation} is used with $\varepsilon=R_T^{-\beta}$ when $\beta<1$; the case $\beta=1$ follows from \eqref{e0}.

By \eqref{12}, \eqref{20} and \eqref{finite-products},
\begin{align*}
 \int_0^T\int_\Omega u_t^2dxdt
 &\leq C(M)\int_0^T\int_\Omega
 (u_{xx}^2+\theta_x^2+v_x^2+u_x^2\theta_x^2+u_x^2v_x^2)dxdt
 \leq C(M).
\end{align*}
Consequently,
\begin{align}
 \left(\frac{v_x}v\right)_t
 &=\frac{u_{xx}}v-\frac{u_xv_x}{v^2},\qquad
 \left(\frac u\mu\right)_t
 =\frac{u_t}\mu-\frac{\alpha u\theta_t}{\mu\theta},\notag\\
 &\int_0^T\int_\Omega
 \left(\left|\left(\frac{v_x}v\right)_t\right|^2
       +\left|\left(\frac u\mu\right)_t\right|^2\right)dxdt\notag\\
 &\quad\leq C(M)\int_0^T\int_\Omega
       (u_{xx}^2+u_x^2v_x^2+u_t^2+u^2\theta_t^2)dxdt
 \leq C(M).\label{cross-chain}
\end{align}
Using $v_t=u_x$ and $\mu_t=\alpha\mu\theta_t/\theta$, we have
\begin{align*}
 \int_\Omega\frac{u_t}\mu\frac{v_x}v dx
 &=\frac d{dt}\int_\Omega\frac{uv_x}{\mu v}dx
   +\int_\Omega\frac{u\mu_t v_x}{\mu^2v}dx
   -\int_\Omega\frac u\mu\left(\frac{u_x}v\right)_xdx\\
 &=\frac d{dt}\int_\Omega\frac{uv_x}{\mu v}dx
   +\alpha\int_\Omega\frac{uv_x\theta_t}{\mu\theta v}dx
   +\int_\Omega\left(\frac u\mu\right)_x\frac{u_x}v dx\\
 &=\frac d{dt}\int_\Omega\frac{uv_x}{\mu v}dx
   +\int_\Omega\frac{u_x^2}{\mu v}dx
   +\alpha\int_\Omega
       \frac{uv_x\theta_t-uu_x\theta_x}{\mu\theta v}dx.
\end{align*}
Multiplying \eqref{2671} by $(\ln v)_x$ and integrating over
$\Omega$, we obtain
\begin{align}
 &\frac d{dt}\int_\Omega
            \left(\frac{v_x^2}{2v^2}-\frac{uv_x}{\mu v}\right)dx
       +R\int_\Omega\frac{\theta v_x^2}{\mu v^3}dx\notag\\
 &\quad=R\int_\Omega\frac{\theta_xv_x}{\mu v^2}dx
       +\int_\Omega\frac{u_x^2}{\mu v}dx
       +\alpha\int_\Omega\frac1{\theta v}
          \left(\frac{uv_x\theta_t}\mu
                -\frac{uu_x\theta_x}\mu
                -\frac{u_xv_x\theta_x}v\right)dx.
 \label{268}
\end{align}
Its time-integrated form is
\begin{align}
 &\int_\Omega\left(\frac{v_x^2}{2v^2}
                         -\frac{uv_x}{\mu v}\right)(x,t)dx
       +R\int_0^t\int_\Omega\frac{\theta v_x^2}{\mu v^3}dxds
 \notag\\
 &=\int_\Omega\left(\frac{v_{0x}^2}{2v_0^2}
                         -\frac{u_0v_{0x}}{\mu_0v_0}\right)dx
       +R\int_0^t\int_\Omega\frac{\theta_xv_x}{\mu v^2}dxds
       +\int_0^t\int_\Omega\frac{u_x^2}{\mu v}dxds\notag\\
 &\quad+\alpha\int_0^t\int_\Omega\frac1{\theta v}
          \left(\frac{uv_x\theta_t}\mu
                -\frac{uu_x\theta_x}\mu
                -\frac{u_xv_x\theta_x}v\right)dxds\notag\\
 &=\int_\Omega\left(\frac{v_{0x}^2}{2v_0^2}
                         -\frac{u_0v_{0x}}{\mu_0v_0}\right)dx
                                                +\sum_{i=1}^3J_i.
 \label{269}
\end{align}
The first two terms satisfy
\begin{align}
 |J_1|
 &\leq\frac R2\int_0^t\int_\Omega
                     \frac{\theta v_x^2}{\mu v^3}dxds
           +C\int_0^t\int_\Omega\theta^{-1}\theta_x^2dxds,\notag\\
 J_2&\leq C\int_0^t\int_\Omega u_x^2dxds.
 \label{log-first-two}
\end{align}
For the remaining terms, \eqref{20} implies
\begin{align}
 \sup_t\|u\|_{L^\infty}^2
 &\leq2\sup_t\|u\|_{L^2}\|u_x\|_{L^2}\leq M,\notag\\
 \int_0^T\|u_x\|_{L^\infty}^2dt
 &\leq2\int_0^T\|u_x\|_{L^2}\|u_{xx}\|_{L^2}dt\leq2M,\notag\\
 |J_3|
 &\leq\frac{\alpha eC_3C_4}{\tilde\mu}
      \sup_s\|u\|_{L^\infty}
        \int_0^t\bigl(\|v_x\|_{L^2}\|\theta_t\|_{L^2}
                         +\|u_x\|_{L^2}\|\theta_x\|_{L^2}\bigr)ds\notag\\
 &\quad+\alpha C_3^2C_4\sup_s\|v_x\|_{L^2}
                            \int_0^t\|u_x\|_{L^\infty}\|\theta_x\|_{L^2}ds
 \notag\\
 &\leq\alpha\left(\frac{2eC_3C_4}{\tilde\mu}
                         +\sqrt2C_3^2C_4\right)M^{3/2}
 \notag\\
 &\leq\alpha\left(\frac{2eC_3C_4}{\tilde\mu}+\sqrt2 C_3^2C_4\right)
                                      (1+M)^2\leq1.\label{log-alpha}
\end{align}
Moreover,
\begin{equation}
 \left|\int_\Omega\frac{uv_x}{\mu v}dx\right|
 \leq\frac14\int_\Omega\frac{v_x^2}{v^2}dx+C\|u\|_{L^2}^2.
 \label{log-endpoint}
\end{equation}
Substituting \eqref{log-first-two}--\eqref{log-endpoint} into
\eqref{269}, we obtain
\begin{align*}
 &\frac14\int_\Omega\frac{v_x^2}{v^2}(x,t)dx
       +\frac R2\int_0^t\int_\Omega\frac{\theta v_x^2}{\mu v^3}dxds\\
 &\quad\leq C+C\int_0^t\int_\Omega
                          (u_x^2+\theta^{-1}\theta_x^2)dxds+|J_3|\\
 &\quad\leq
 C\bigl(1+R_T^{\beta/(1+\beta)}
                +R_T^{\beta(1-\beta)/(1+\beta)}\bigr)\\
 &\sup_{t\leq T}\|v_x(t)\|_{L^2}^2+
   \int_0^T\int_\Omega\bigl((1+\theta)v_x^2+u_x^2\bigr)dxdt\\
 &\quad\leq
 CR_T^{\beta/(1+\beta)}
\end{align*}
Here we used \eqref{ux-low}, \eqref{theta-lower} and
$\beta(1-\beta)\leq\beta$ for $0\leq\beta\leq1$.
\qed

We now estimate the first-order derivative of velocity.
\begin{lemma}\label{lm26}
Under \eqref{volume-smallness} and \eqref{derivative-smallness},
for $0\leq\beta\leq1$,
\begin{equation}\label{273}
 \sup_{0\leq t\leq T}\|u_x(t)\|_{L^2}^2
 +\int_0^T\|(u_{xx},u_t,\theta_x)\|_{L^2}^2dt
 \leq C R_T^{3\beta/(1+\beta)}.
\end{equation}
\end{lemma}
\noindent{\it Proof.}
First, we rewrite \eqref{12} as
\begin{equation}\label{274}
 u_t-\frac\mu v u_{xx}
 =-\frac{R\theta_x}v+\frac{R\theta v_x}{v^2}
                  +\frac{\mu_xu_x}v-\frac{\mu u_xv_x}{v^2}.
\end{equation}
Multiplying by $-u_{xx}$ and integrating by parts,
\begin{align}
 &\frac12\|u_x(t)\|_{L^2}^2
       +\int_0^t\int_\Omega\frac\mu v u_{xx}^2dxds\notag\\
 &\qquad=\frac12\|u_{0x}\|_{L^2}^2
       +R\int_0^t\int_\Omega\frac{\theta_xu_{xx}}v dxds
       -R\int_0^t\int_\Omega\frac{\theta v_xu_{xx}}{v^2}dxds\notag\\
 &\qquad\quad-\int_0^t\int_\Omega\frac{\mu_xu_xu_{xx}}v dxds
       +\int_0^t\int_\Omega\frac{\mu u_xv_xu_{xx}}{v^2}dxds\notag\\
 &\qquad\leq C+\frac14\int_0^t\int_\Omega\frac\mu v u_{xx}^2dxds\notag\\
 &\qquad\quad
       +C\int_0^t\int_\Omega\theta_x^2dxds
       +C\int_0^t\int_\Omega\theta^2v_x^2dxds\notag\\
 &\qquad\quad+C\int_0^t\int_\Omega v_x^2u_x^2dxds
       +C\int_0^t\int_\Omega\mu_x^2u_x^2dxds\notag\\
 &\qquad=C+\frac14\int_0^t\int_\Omega\frac\mu v u_{xx}^2dxds
       +\sum_{i=1}^4K_i.
 \label{275}
\end{align}
The first two integrals on the right satisfy
\begin{align}
 K_1=C\int_0^t\|\theta_x\|_{L^2}^2ds
 &\leq C\int_0^T\int_\Omega\theta^\beta\theta_x^2dxdt
 \leq
 C R_T^{2\beta/(1+\beta)}\notag\\
 K_2=C\int_0^t\int_\Omega\theta^2v_x^2dxds
 &\leq C\int_0^T\int_\Omega\theta v_x^2dxdt
       +C\sup_t\|v_x\|_{L^2}^2
                     \int_0^T\|(\theta-3/2)_+\|_{L^\infty}^2dt\notag\\
 &\leq
 C R_T^{(2\beta-\beta^2)/(1+\beta)}.\label{pressure-gradient}
\end{align}
Since $\mu_x=\alpha\mu\theta_x/\theta$,
\begin{align}
 \sup_t\|\mu_x\|_{L^2}^2
 &\leq\alpha^2(e\tilde\mu C_4)^2
                  \sup_t\|\theta_x\|_{L^2}^2
   \leq\alpha^2(e\tilde\mu C_4)^2M\leq1,
 \notag\\
 K_3+K_4
 &\leq2C\left(1+\sup_t\|v_x\|_{L^2}^2\right)
                    \int_0^T\|u_x\|_{L^2}\|u_{xx}\|_{L^2}dt\notag\\
 &\leq\frac14\int_0^T\int_\Omega\frac\mu v u_{xx}^2dxdt
       +C\left(1+\sup_t\|v_x\|_{L^2}^2\right)^2
                             \int_0^T\|u_x\|_{L^2}^2dt\notag\\
 &\leq\frac14\int_0^T\int_\Omega\frac\mu v u_{xx}^2dxdt
       +C R_T^{3\beta/(1+\beta)}.\label{momentum-products}
\end{align}
Substitution of \eqref{pressure-gradient} and \eqref{momentum-products}
into \eqref{275} gives
\begin{align*}
 &\frac12\sup_{t\leq T}\|u_x(t)\|_{L^2}^2
       +\frac12\int_0^T\int_\Omega\frac\mu v u_{xx}^2dxdt\\
 &\quad\leq C+C\int_0^T\int_\Omega
                         (\theta_x^2+\theta^2v_x^2)dxdt
       +C\left(1+\sup_t\|v_x\|_{L^2}^2\right)^2
                              \int_0^T\|u_x\|_{L^2}^2dt\\
 &\quad\leq
 C\bigl(R_T^{2\beta/(1+\beta)}
     +R_T^{(2\beta-\beta^2)/(1+\beta)}
     +R_T^{3\beta/(1+\beta)}\bigr)\\
 &\quad\leq
 CR_T^{3\beta/(1+\beta)}
\end{align*}
Finally, squaring \eqref{274},
\begin{align}
 \int_0^T\|u_t\|_{L^2}^2dt
 &\leq C\int_0^T\int_\Omega
      \bigl(u_{xx}^2+\theta_x^2+\theta^2v_x^2
                         +v_x^2u_x^2+\mu_x^2u_x^2\bigr)dxdt\notag\\
 &\leq
 C R_T^{3\beta/(1+\beta)}
\end{align}
This proves \eqref{273}.
\qed

The next lemma gives the upper and lower bounds of temperature.
\begin{lemma}\label{lm27}
Under \eqref{volume-smallness} and \eqref{derivative-smallness}, for
$0\leq\beta\leq1$,
\begin{equation}\label{281}
 \sup_{0\leq t\leq T}
   \left(\|\theta(t)\|_{L^\infty}+\|\theta(t)^{-1}\|_{L^\infty}
                         +\|\theta^\beta\theta_x(t)\|_{L^2}^2\right)
       +\int_0^T\|\theta^{\beta/2}\theta_t\|_{L^2}^2dt\leq C.
\end{equation}

\end{lemma}
\noindent{\it Proof.}
Multiplying \eqref{13a} by $\theta^\beta\theta_t$ and integrating by parts, we obtain
\begin{align}
 &\frac{\tilde\kappa}2\frac d{dt}
             \int_\Omega\frac{\theta^{2\beta}\theta_x^2}v dx
                          +c_v\int_\Omega\theta^\beta\theta_t^2dx
 \notag\\
 &\quad=\tilde\mu\int_\Omega
                   \frac{\theta^{\alpha+\beta}u_x^2\theta_t}v dx
       -R\int_\Omega\frac{\theta^{\beta+1}u_x\theta_t}v dx
       -\frac{\tilde\kappa}2\int_\Omega
                   \frac{\theta^{2\beta}\theta_x^2u_x}{v^2}dx
 \notag\\
 &\quad\leq\frac{c_v}2\int_\Omega\theta^\beta\theta_t^2dx
       +C\int_\Omega\theta^\beta u_x^4dx
       +C\int_\Omega\theta^{\beta+2}u_x^2dx
       +C\|u_x\|_{L^\infty}
                    \int_\Omega\frac{\theta^{2\beta}\theta_x^2}v dx.
 \label{282}
\end{align}
The initial term satisfies
\begin{equation}
 \int_\Omega\frac{\theta_0^{2\beta}\theta_{0x}^2}{v_0}dx
 \leq(\inf v_0)^{-1}\|\theta_0\|_{L^\infty}^{2\beta}
                                      \|\theta_{0x}\|_{L^2}^2\leq C.
\end{equation}
Integrating \eqref{282} gives
\begin{align}
 &\int_\Omega\frac{\theta^{2\beta}\theta_x^2}v(x,t)dx
       +\frac{c_v}{\tilde\kappa}
                         \int_0^t\int_\Omega\theta^\beta\theta_t^2dxds
 \notag\\
 &\quad\leq C+C\int_0^t\int_\Omega\theta^\beta u_x^4dxds
       +C\int_0^t\int_\Omega\theta^{\beta+2}u_x^2dxds
       +C\int_0^t\|u_x\|_{L^\infty}
                    \int_\Omega\frac{\theta^{2\beta}\theta_x^2}v dxds
 \notag\\
 &\quad=C+\sum_{i=1}^3L_i.
 \label{thermal-three}
\end{align}
For every $\varepsilon>0$,
\begin{align}
 \|u_x\|_{L^6}^3
 &\leq\|u_x\|_{L^\infty}^2\|u_x\|_{L^2}
 \leq2\|u_x\|_{L^2}^2\|u_{xx}\|_{L^2},\notag\\
 \int_0^T\int_\Omega\theta^\beta u_x^4dxdt
 &\leq\int_0^T\|\theta^\beta u_x\|_{L^2}\|u_x\|_{L^6}^3dt
 \leq2\sup_t\|u_x\|_{L^2}^2
           \int_0^T\|\theta^\beta u_x\|_{L^2}\|u_{xx}\|_{L^2}dt\notag\\
 &\leq\sup_t\|u_x\|_{L^2}^2
       \left(\varepsilon\int_0^T\|\theta^\beta u_x\|_{L^2}^2dt
                   +\varepsilon^{-1}\int_0^T\|u_{xx}\|_{L^2}^2dt\right).
 \label{283}
\end{align}
Taking $\varepsilon=1$ in \eqref{283}, we have
\begin{align}
 \int_0^T\|\theta^\beta u_x\|_{L^2}^2dt
 &\leq R_T^{2\beta}\int_0^T\|u_x\|_{L^2}^2dt
                         \leq C R_T^{2\beta+\beta/(1+\beta)},\notag\\
 L_1&\leq C R_T^{3\beta/(1+\beta)}
       \left(R_T^{2\beta+\beta/(1+\beta)}
                           +R_T^{3\beta/(1+\beta)}\right)
 \leq C R_T^{2\beta+4\beta/(1+\beta)}.
 \label{quartic-small-beta}
\end{align}
For the pressure term, Lemma~\ref{lm24} yields
\begin{equation}\label{284}
 L_2\leq C R_T^{\beta+1}\int_0^T\int_\Omega\theta u_x^2dxdt
 \leq
 C R_T^{\beta+1+2\beta/(1+\beta)}
\end{equation}
For $L_3$, Young's inequality gives
\begin{align}
 C\|u_x\|_{L^\infty}
 &\leq C\|u_x\|_{L^2}^{1/2}\|u_{xx}\|_{L^2}^{1/2}\notag\\
 &\leq\varepsilon_1\|u_x\|_{L^2}^2
       +\varepsilon_2\|u_{xx}\|_{L^2}^2
       +C\varepsilon_1^{-1/2}\varepsilon_2^{-1/2},\notag\\
 L_3
 &\leq\int_0^t
      \bigl(\varepsilon_1\|u_x\|_{L^2}^2+\varepsilon_2\|u_{xx}\|_{L^2}^2\bigr)
                   \int_\Omega\frac{\theta^{2\beta}\theta_x^2}v dxds
 \notag\\
 &\quad+C\varepsilon_1^{-1/2}\varepsilon_2^{-1/2}
                       \int_0^t\int_\Omega\theta^{2\beta}\theta_x^2dxds,
 \label{285}\\
 \int_0^T\int_\Omega\theta^{2\beta}\theta_x^2dxdt
 &\leq R_T^\beta\int_0^T\int_\Omega\theta^\beta\theta_x^2dxdt
 \leq
 C R_T^{\beta+2\beta/(1+\beta)}\notag
\end{align}
Choose
\begin{equation}
 (\varepsilon_1,\varepsilon_2)=
 (R_T^{-\beta/(1+\beta)},R_T^{-3\beta/(1+\beta)})
\end{equation}
Then
\begin{align}
 &\int_0^T
      \bigl(\varepsilon_1\|u_x\|_{L^2}^2+\varepsilon_2\|u_{xx}\|_{L^2}^2\bigr)dt
 \leq C,\notag\\
 &\varepsilon_1^{-1/2}\varepsilon_2^{-1/2}
                  \int_0^T\int_\Omega\theta^{2\beta}\theta_x^2dxdt
 \leq
 C R_T^{\beta+4\beta/(1+\beta)}
\end{align}
Substituting \eqref{283}--\eqref{285} into \eqref{thermal-three}
and applying Gronwall's inequality, we obtain
\begin{align*}
 &\int_\Omega\frac{\theta^{2\beta}\theta_x^2}v(x,t)dx
       +\frac{c_v}{\tilde\kappa}
                         \int_0^t\int_\Omega\theta^\beta\theta_t^2dxds\\
 &\quad\leq C\left[\int_\Omega
                         \frac{\theta_0^{2\beta}\theta_{0x}^2}{v_0}dx
       +\int_0^T\int_\Omega\theta^\beta u_x^4dxds
       +\int_0^T\int_\Omega\theta^{\beta+2}u_x^2dxds\right.\\
 &\hspace{24mm}\left.
       +\varepsilon_1^{-1/2}\varepsilon_2^{-1/2}
                        \int_0^T\int_\Omega\theta^{2\beta}\theta_x^2dxds
             \right]\\
 &\hspace{17mm}\times\exp\left\{
        \int_0^T\bigl(\varepsilon_1\|u_x\|_{L^2}^2
                          +\varepsilon_2\|u_{xx}\|_{L^2}^2\bigr)ds\right\}\\
 &\quad\leq C\left[1+\int_0^T\int_\Omega
                          (\theta^\beta u_x^4+\theta^{\beta+2}u_x^2)dxds
       +\varepsilon_1^{-1/2}\varepsilon_2^{-1/2}
                        \int_0^T\int_\Omega\theta^{2\beta}\theta_x^2dxds
             \right].
\end{align*}
Taking the supremum and using \eqref{quartic-small-beta},
\eqref{284},
\begin{align}
 &\sup_{0\leq t\leq T}\|\theta^\beta\theta_x(t)\|_{L^2}^2
                 +\int_0^T\|\theta^{\beta/2}\theta_t\|_{L^2}^2dt\notag\\
 &\qquad\leq
 C\left(R_T^{2\beta+4\beta/(1+\beta)}
              +R_T^{\beta+1+2\beta/(1+\beta)}\right).\label{286}
\end{align}

For $x\in[k,k+1]$, \eqref{v0}--\eqref{v1} imply
\begin{align}
 \theta^{\beta+3/2}(x,t)
 &\leq\alpha_2^{\beta+3/2}
       +(\beta+3/2)\int_k^{k+1}\theta^{1/2}
                                  |\theta^\beta\theta_x|dx\notag\\
 &\leq\alpha_2^{\beta+3/2}
       +(\beta+3/2)\alpha_2^{1/2}
                     \|\theta^\beta\theta_x(t)\|_{L^2(k,k+1)},\notag\\
 (R_T-1)^{2\beta+3}
 &\leq2\alpha_2^{2\beta+3}
       +2(\beta+3/2)^2\alpha_2
                        \sup_t\|\theta^\beta\theta_x(t)\|_{L^2}^2.
 \label{peak-square}
\end{align}
The exponents on the right of \eqref{286} are strictly smaller than
$2\beta+3$. Indeed,
\begin{align}
 2\beta+3-\left(2\beta+\frac{4\beta}{1+\beta}\right)
 &=\frac{3-\beta}{1+\beta}>0,\notag\\
 2\beta+3-\left(\beta+1+\frac{2\beta}{1+\beta}\right)
 &=\frac{\beta^2+\beta+2}{1+\beta}>0
                         \qquad(0\leq\beta\leq1).\notag
\end{align}
Thus \eqref{286} and \eqref{peak-square} give
\begin{equation}
 R_T^{2\beta+3}
 \leq C+C\sup_t\|\theta^\beta\theta_x(t)\|_{L^2}^2
 \leq\frac12 R_T^{2\beta+3}+C,\qquad R_T\leq C.
\end{equation}
Substituting this in \eqref{286} and using \eqref{theta-lower}
proves \eqref{281}.

Finally, multiplying \eqref{13a} by $(\theta-1)^5$ yields
\begin{align}
 &\frac{c_v}6\frac d{dt}\int_\Omega(\theta-1)^6dx\notag\\*
 &\quad=-5\tilde\kappa\int_\Omega
              \frac{(\theta-1)^4\theta^\beta\theta_x^2}v dx
       -R\int_\Omega\frac{\theta(\theta-1)^5u_x}v dx\notag\\
 &\qquad+\int_\Omega\frac{\mu(\theta-1)^5u_x^2}v dx,\notag\\
 &R\left|\int_\Omega\frac{\theta(\theta-1)^5u_x}v dx\right|
 \leq C\|u_x\|_{L^2}
                       \left(\int_\Omega|\theta-1|^{10}dx\right)^{1/2}
 \notag\\
 &\qquad\leq C\|u_x\|_{L^2}^2+C\int_\Omega|\theta-1|^{10}dx
 \leq C\|u_x\|_{L^2}^2+C\|\theta-1\|_{L^\infty}^4
                                      \int_\Omega(\theta-1)^6dx
 \notag\\
 &\qquad\leq C\|u_x\|_{L^2}^2+C\int_\Omega(\theta-1)^6dx,\notag\\
 &5\tilde\kappa\int_\Omega
                   \frac{(\theta-1)^4\theta^\beta\theta_x^2}v dx
       +\left|\int_\Omega\frac{\mu(\theta-1)^5u_x^2}v dx\right|
 \notag\\
 &\qquad\leq C\|\theta-1\|_{L^\infty}^4\|\theta_x\|_{L^2}^2
       +C\|\theta-1\|_{L^\infty}^5\|u_x\|_{L^2}^2
 \leq C(\|\theta_x\|_{L^2}^2+\|u_x\|_{L^2}^2),\notag\\
 &\int_\Omega(\theta-1)^6dx
 \leq4\|\theta-1\|_{L^2}^4\|\theta_x\|_{L^2}^2
                                      \leq C\|\theta_x\|_{L^2}^2,\notag\\
 &\left|\frac d{dt}\int_\Omega(\theta-1)^6dx\right|\notag\\
 &\qquad\leq C\left(\|\theta_x\|_{L^2}^2+\|u_x\|_{L^2}^2
                           +\int_\Omega(\theta-1)^6dx\right)
 \notag\\
 &\qquad\leq C\bigl(\|\theta_x\|_{L^2}^2+\|u_x\|_{L^2}^2\bigr),\notag\\
 &\int_0^T\left\{\int_\Omega(\theta-1)^6dx
               +\left|\frac d{dt}\int_\Omega(\theta-1)^6dx\right|\right\}dt
 \leq C.
 \label{sixth-temperature}
\end{align}
Moreover,
\begin{align}
 \|\theta^{\beta+1}-1\|_{L^\infty}^2
 &\leq2(\beta+1)^{1/2}
       \left(\int_\Omega|\theta^{\beta+1}-1|^6dx\right)^{1/4}
                                \|\theta^\beta\theta_x\|_{L^2}^{1/2}
 \notag\\
 &\leq C\left(\int_\Omega(\theta-1)^6dx\right)^{1/4}.
 \label{sixth-peak}
\end{align}
\qed

\subsubsection*{2.1.2. Case $\beta>1$}
We first estimate the negative powers of $\theta$.
\begin{lemma}\label{lm240}
Under the assumptions of Lemma~\ref{lm23}, for every $p\geq1$,
\begin{equation}\label{2440}
 \int_0^T\int_\Omega
 (\theta^{\beta-p-1}\theta_x^2+\theta^{-p}u_x^2)dxdt\leq C(p).
\end{equation}
\end{lemma}
\noindent{\it Proof.}
For $p=1$, this is a consequence of \eqref{e0}, \eqref{ga} and
\eqref{vvv}. For $p>1$, multiplying \eqref{13a} by
$-(\theta^{-p}-4)_+$ and integrating, we obtain
\begin{align}
 &c_v\int_\Omega\int_\theta^{4^{-1/p}}(s^{-p}-4)_+dsdx
 +p\tilde\kappa\int_0^t\int_{(\theta<4^{-1/p})(s)}
                    \frac{\theta^{\beta-p-1}\theta_x^2}v dxds\notag\\
 &\quad+\int_0^t\int_\Omega\frac{\mu u_x^2}v
                    (\theta^{-p}-4)_+dxds\notag\\
 &=c_v\int_\Omega\int_{\theta_0}^{4^{-1/p}}(s^{-p}-4)_+dsdx
 +R\int_0^t\int_\Omega\frac{\theta u_x}v
                    (\theta^{-p}-4)_+dxds.\label{247a}
\end{align}
By \eqref{theta-lower},
\begin{align}
 &\int_0^T\int_\Omega
      (\theta^{\beta-p-1}\theta_x^2+\theta^{-p}u_x^2)dxdt\notag\\
 &\quad\leq C_4^{p-1}\int_0^T\int_\Omega
      (\theta^{\beta-2}\theta_x^2+\theta^{-1}u_x^2)dxdt
 \leq C(p),\notag\\
 &\left|R\int_0^t\int_\Omega\frac{\theta u_x}v
                    (\theta^{-p}-4)_+dxds\right|\notag\\
 &\quad\leq C(p)\int_0^t\|\theta^{-p/2}u_x\|_{L^2}
             \|(1-4\theta^p)_+\|_{L^\infty}
             |(\theta<4^{-1/p})(s)|^{1/2}ds\notag\\
 &\quad\leq\varepsilon\int_0^t\int_\Omega\theta^{-p}u_x^2dxds
       +C(p,\varepsilon)\int_0^t\|(1-4\theta^p)_+\|_{L^\infty}^2ds,
 \notag\\
 &\int_0^t\|(1-4\theta^p)_+\|_{L^\infty}^2ds
 \leq C(p)\int_0^t\int_\Omega\theta^{\beta-2}\theta_x^2dxds
 \leq C(p).\label{250a}
\end{align}
Indeed, \eqref{e0} and \eqref{theta-lower} imply
\begin{align*}
 \|(1-4\theta^p)_+\|_{L^\infty}^2
 &\leq16p^2\left(\int_{(\theta<4^{-1/p})(t)}
                            \theta^{p-1}|\theta_x|dx\right)^2\\
 &\leq C(p)|(\theta<4^{-1/p})(t)|
                   \int_{(\theta<4^{-1/p})(t)}
                            \theta^{\beta-2}\theta_x^2dx
 \leq C(p)V(t).
\end{align*}
This proves \eqref{2440}.
\qed

\begin{lemma}\label{lm28}
Under \eqref{volume-smallness} and \eqref{derivative-smallness},
for $\beta>1$,
\begin{align}
 &\sup_{0\leq t\leq T}\int_\Omega
       \bigl((\theta-1)^2+\theta u^2\bigr)dx\leq C R_T^\beta,
 \label{2110}\\
 &\int_0^T\int_\Omega(\theta u_x^2+\theta^\beta\theta_x^2)dxdt
 \leq C R_T^{\min\{\beta,2\}},\label{21100}\\
 &\sup_{0\leq t\leq T}\int_\Omega u^4dx
 +\int_0^T\int_\Omega u^2u_x^2dxdt
 +\int_0^T\|(\theta-3/2)_+\|_{L^\infty}^2dt\leq C.\label{2111}
\end{align}
\begin{align}
 &\sup_{t\leq T}\|v_x(t)\|_{L^2}^2
 +\int_0^T\int_\Omega((1+\theta)v_x^2+u_x^2)dxdt
 \leq CR_T^{\min\{\beta/2,1\}},\label{2112}\\
 &\sup_{t\leq T}\|u_x(t)\|_{L^2}^2
 +\int_0^T\|(u_{xx},u_t,\theta_x)\|_{L^2}^2dt
 \leq CR_T^{\min\{3\beta/2,3\}}.\label{2113}
\end{align}
\end{lemma}
\noindent{\it Proof.}
For $\beta>1$, the calculations \eqref{246}--\eqref{hot-peak}
remain valid. By \eqref{2440} with $p=\beta$,
\begin{equation}\label{26210}
 \int_0^T\|(\theta-3/2)_+\|_{L^\infty}^2dt
 \leq C\int_0^T\int_\Omega\theta^{-1}\theta_x^2dxdt\leq C.
\end{equation}
Hence \eqref{quartic-bound} and \eqref{low-order-closure} give
\begin{align}
 &\sup_{t\leq T}\int_\Omega u^4dx
       +\int_0^T\int_\Omega u^2u_x^2dxdt\leq C,\notag\\
 &\sup_{t\leq T}\int_\Omega((\theta-1)^2+\theta u^2)dx
       +\int_0^T\int_\Omega(\theta u_x^2+\theta^\beta\theta_x^2)dxdt
       \leq CR_T^\beta,\notag\\
 &\int_0^T\int_\Omega(\theta u_x^2+\theta^\beta\theta_x^2)dxdt
 \leq CR_T^2\int_0^T V(t)dt\leq CR_T^2.\label{2114}
\end{align}
Next, using \eqref{269}, \eqref{log-alpha} and \eqref{log-endpoint},
we obtain
\begin{align}
 &\int_0^T\int_\Omega u_x^2dxdt
 \leq\left(\int_0^T\int_\Omega\theta u_x^2dxdt\right)^{1/2}
       \left(\int_0^T\int_\Omega\theta^{-1}u_x^2dxdt\right)^{1/2}
 \leq CR_T^{\min\{\beta/2,1\}},\notag\\
 &\frac14\int_\Omega\frac{v_x^2}{v^2}dx
 +\frac R2\int_0^t\int_\Omega\frac{\theta v_x^2}{\mu v^3}dxds\notag\\
 &\quad\leq C+C\int_0^t\int_\Omega
                   (u_x^2+\theta^{-1}\theta_x^2)dxds
 \leq CR_T^{\min\{\beta/2,1\}},\notag\\
 &\sup_{t\leq T}\|v_x(t)\|_{L^2}^2
 +\int_0^T\int_\Omega((1+\theta)v_x^2+u_x^2)dxdt
 \leq CR_T^{\min\{\beta/2,1\}}.\label{2115}
\end{align}
The last inequality also uses \eqref{theta-lower}.
By \eqref{275},
\begin{align}
 &\frac12\|u_x(t)\|_{L^2}^2
 +\frac34\int_0^t\int_\Omega\frac\mu v u_{xx}^2dxds\notag\\
 &\quad\leq C+C\int_0^t\int_\Omega\theta_x^2dxds
       +C\int_0^t\int_\Omega\theta^2v_x^2dxds\notag\\
 &\qquad+C\int_0^t\int_\Omega v_x^2u_x^2dxds
       +C\int_0^t\int_\Omega\mu_x^2u_x^2dxds
 =C+\sum_{i=1}^4M_i.\label{21160}
\end{align}
By \eqref{theta-lower}, \eqref{2114}, \eqref{26210} and \eqref{2115},
\begin{align}
 M_1&\leq C\int_0^T\int_\Omega\theta^\beta\theta_x^2dxdt
 \leq CR_T^{\min\{\beta,2\}},\label{2770}\\
 M_2&\leq C\int_0^T\int_\Omega\theta v_x^2dxdt
       +C\sup_t\|v_x\|_{L^2}^2\int_0^T\|(\theta-3/2)_+\|_{L^\infty}^2dt
 \leq CR_T^{\min\{\beta/2,1\}},\notag\\
 M_3+M_4
 &\leq2C(1+\sup_t\|v_x\|_{L^2}^2)
                     \int_0^T\|u_x\|_{L^2}\|u_{xx}\|_{L^2}dt\notag\\
 &\leq\frac14\int_0^T\int_\Omega\frac\mu v u_{xx}^2dxdt
       +C(1+\sup_t\|v_x\|_{L^2}^2)^2\int_0^T\|u_x\|_{L^2}^2dt\notag\\
 &\leq\frac14\int_0^T\int_\Omega\frac\mu v u_{xx}^2dxdt
       +CR_T^{\min\{3\beta/2,3\}}.\label{2116}
\end{align}
Here $\sup_t\|\mu_x\|_{L^2}^2\leq\alpha^2(e\tilde\mu C_4)^2M\leq1$.
Substitution into \eqref{21160} and then \eqref{274} yields
\begin{align}
 &\sup_{t\leq T}\|u_x(t)\|_{L^2}^2
       +\int_0^T\|(u_{xx},u_t,\theta_x)\|_{L^2}^2dt\notag\\
 &\quad\leq C\bigl(R_T^{\min\{\beta,2\}}
             +R_T^{\min\{\beta/2,1\}}
             +R_T^{\min\{3\beta/2,3\}}\bigr)
 \leq CR_T^{\min\{3\beta/2,3\}}.\label{2117}
\end{align}
\qed

We next estimate $\theta^\beta\theta_x$ for $\beta>1$.
\begin{lemma}\label{lm29}
Under \eqref{volume-smallness} and \eqref{derivative-smallness}, for
$\beta>1$,
\begin{equation}\label{2118}
 \sup_{0\leq t\leq T}
   \left(\|\theta(t)\|_{L^\infty}+\|\theta(t)^{-1}\|_{L^\infty}
                         +\|\theta^\beta\theta_x(t)\|_{L^2}^2\right)
       +\int_0^T\|\theta^{\beta/2}\theta_t\|_{L^2}^2dt\leq C.
\end{equation}

\end{lemma}
\noindent{\it Proof.}
As in \eqref{282}--\eqref{thermal-three},
\begin{align}
 &\int_\Omega\frac{\theta^{2\beta}\theta_x^2}v dx
       +\frac{c_v}{\tilde\kappa}\int_0^t\int_\Omega\theta^\beta\theta_t^2dxds
 \notag\\
 &\quad\leq C+C\int_0^t\int_\Omega\theta^\beta u_x^4dxds
       +C\int_0^t\int_\Omega\theta^{\beta+2}u_x^2dxds\notag\\
 &\qquad+C\int_0^t\|u_x\|_{L^\infty}
                \int_\Omega\frac{\theta^{2\beta}\theta_x^2}v dxds
 =C+\sum_{i=1}^3N_i.\label{2820}
\end{align}
For $\beta>1$, \eqref{v1} gives, for any $\eta>0$,
\begin{align}
 \max_{[k,k+1]}\theta^{2\beta}
 &\leq\alpha_2^{2\beta}
           +2\beta\int_k^{k+1}\theta^{2\beta-1}|\theta_x|dx\notag\\
 &\leq\alpha_2^{2\beta}
       +\eta\int_k^{k+1}\theta^{2\beta+1}dx
       +\frac{\beta^2}{\eta}
                         \int_k^{k+1}\theta^{2\beta-3}\theta_x^2dx
 \notag\\
 &\leq\alpha_2^{2\beta}
       +\eta\alpha_2\max_{[k,k+1]}\theta^{2\beta}
       +\frac{\beta^2}{\eta}R_T^{\beta-1}
                         \int_k^{k+1}\theta^{\beta-2}\theta_x^2dx,
 \notag\\
 \|\theta^{2\beta}\|_{L^\infty}
 &\leq C+C R_T^{\beta-1}
                         \int_\Omega\theta^{\beta-2}\theta_x^2dx
 \qquad\left(\eta=\frac1{2\alpha_2}\right),\notag\\
 \int_0^T\|\theta^\beta u_x\|_{L^2}^2dt
 &\leq C\int_0^T\|u_x\|_{L^2}^2dt
       +C R_T^{\beta-1}\sup_t\|u_x\|_{L^2}^2
                       \int_0^T\int_\Omega\theta^{\beta-2}\theta_x^2dxdt
 \notag\\
 &\leq C R_T^{\min\{\beta/2,1\}}
                   +C R_T^{\beta-1+\min\{3\beta/2,3\}}.
 \label{large-beta-local}
\end{align}
Taking $\varepsilon=R_T^{(1-\beta)/2}$ in \eqref{283}, we find
\begin{align}
 N_1
 &\leq C R_T^{\min\{3\beta/2,3\}}
       \left\{R_T^{(1-\beta)/2+\min\{\beta/2,1\}}
                    +R_T^{(\beta-1)/2+\min\{3\beta/2,3\}}\right\}
 \notag\\
 &\leq C R_T^{\min\{(7\beta-1)/2,(\beta+11)/2\}}.
 \label{quartic-large-beta}
\end{align}
For the pressure term, Lemma~\ref{lm28} yields
\begin{equation}\label{2840}
 N_2\leq C R_T^{\beta+1}\int_0^T\int_\Omega\theta u_x^2dxdt
 \leq
 C R_T^{\beta+1+\min\{\beta,2\}}
\end{equation}
For $N_3$, Young's inequality gives
\begin{align}
 C\|u_x\|_{L^\infty}
 &\leq C\|u_x\|_{L^2}^{1/2}\|u_{xx}\|_{L^2}^{1/2}\notag\\
 &\leq\varepsilon_1\|u_x\|_{L^2}^2
       +\varepsilon_2\|u_{xx}\|_{L^2}^2
       +C\varepsilon_1^{-1/2}\varepsilon_2^{-1/2},\notag\\
 N_3
 &\leq\int_0^t
      \bigl(\varepsilon_1\|u_x\|_{L^2}^2+\varepsilon_2\|u_{xx}\|_{L^2}^2\bigr)
                   \int_\Omega\frac{\theta^{2\beta}\theta_x^2}v dxds
 \notag\\
 &\quad+C\varepsilon_1^{-1/2}\varepsilon_2^{-1/2}
                       \int_0^t\int_\Omega\theta^{2\beta}\theta_x^2dxds,
 \label{2850}\\
 \int_0^T\int_\Omega\theta^{2\beta}\theta_x^2dxdt
 &\leq R_T^\beta\int_0^T\int_\Omega\theta^\beta\theta_x^2dxdt
 \leq
 C R_T^{\beta+\min\{\beta,2\}}\notag
\end{align}
Choose
\begin{equation}
 (\varepsilon_1,\varepsilon_2)=
 (R_T^{-\min\{\beta/2,1\}},R_T^{-\min\{3\beta/2,3\}})
\end{equation}
Then
\begin{align}
 &\int_0^T
      \bigl(\varepsilon_1\|u_x\|_{L^2}^2+\varepsilon_2\|u_{xx}\|_{L^2}^2\bigr)dt
 \leq C,\notag\\
 &\varepsilon_1^{-1/2}\varepsilon_2^{-1/2}
                  \int_0^T\int_\Omega\theta^{2\beta}\theta_x^2dxdt
 \leq
 C R_T^{\min\{3\beta,\beta+4\}}
\end{align}
Substituting \eqref{283}, \eqref{2840} and \eqref{2850} into \eqref{2820}
and applying Gronwall's inequality, we obtain
\begin{align*}
 &\int_\Omega\frac{\theta^{2\beta}\theta_x^2}v(x,t)dx
       +\frac{c_v}{\tilde\kappa}
                         \int_0^t\int_\Omega\theta^\beta\theta_t^2dxds\\
 &\quad\leq C\left[\int_\Omega
                         \frac{\theta_0^{2\beta}\theta_{0x}^2}{v_0}dx
       +\int_0^T\int_\Omega\theta^\beta u_x^4dxds
       +\int_0^T\int_\Omega\theta^{\beta+2}u_x^2dxds\right.\\
 &\hspace{24mm}\left.
       +\varepsilon_1^{-1/2}\varepsilon_2^{-1/2}
                        \int_0^T\int_\Omega\theta^{2\beta}\theta_x^2dxds
             \right]\\
 &\hspace{17mm}\times\exp\left\{
        \int_0^T\bigl(\varepsilon_1\|u_x\|_{L^2}^2
                          +\varepsilon_2\|u_{xx}\|_{L^2}^2\bigr)ds\right\}\\
 &\quad\leq C\left[1+\int_0^T\int_\Omega
                          (\theta^\beta u_x^4+\theta^{\beta+2}u_x^2)dxds
       +\varepsilon_1^{-1/2}\varepsilon_2^{-1/2}
                        \int_0^T\int_\Omega\theta^{2\beta}\theta_x^2dxds
             \right].
\end{align*}
Taking the supremum and using \eqref{quartic-large-beta} and \eqref{2840},
\begin{align}
 &\sup_{0\leq t\leq T}\|\theta^\beta\theta_x(t)\|_{L^2}^2
                 +\int_0^T\|\theta^{\beta/2}\theta_t\|_{L^2}^2dt\notag\\
 &\qquad\leq
 C\left(R_T^{\min\{(7\beta-1)/2,(\beta+11)/2\}}
              +R_T^{\min\{3\beta,\beta+4\}}\right).\label{2860}
\end{align}

For $x\in[k,k+1]$, \eqref{v0}--\eqref{v1} imply
\begin{align}
 \theta^{\beta+3/2}(x,t)
 &\leq\alpha_2^{\beta+3/2}
       +(\beta+3/2)\int_k^{k+1}\theta^{1/2}
                                  |\theta^\beta\theta_x|dx\notag\\
 &\leq\alpha_2^{\beta+3/2}
       +(\beta+3/2)\alpha_2^{1/2}
                     \|\theta^\beta\theta_x(t)\|_{L^2(k,k+1)},\notag\\
 (R_T-1)^{2\beta+3}
 &\leq2\alpha_2^{2\beta+3}
       +2(\beta+3/2)^2\alpha_2
                        \sup_t\|\theta^\beta\theta_x(t)\|_{L^2}^2.
 \label{peak-square0}
\end{align}
The exponents on the right of \eqref{2860} are strictly smaller than
$2\beta+3$. Indeed,
\begin{align}
 2\beta+3-\frac{7\beta-1}2&=\frac{7-3\beta}2>0,\notag\\
 2\beta+3-3\beta&=3-\beta>0
                         \qquad(1<\beta\leq2),\notag\\
 2\beta+3-\frac{\beta+11}2&=\frac{3\beta-5}2>0,\notag\\
 2\beta+3-(\beta+4)&=\beta-1>0
                         \qquad(\beta\geq2).
\end{align}
Thus \eqref{2860} and \eqref{peak-square0} give
\begin{equation}
 R_T^{2\beta+3}
 \leq C+C\sup_t\|\theta^\beta\theta_x(t)\|_{L^2}^2
 \leq\frac12 R_T^{2\beta+3}+C,\qquad R_T\leq C.
\end{equation}
Substituting this in \eqref{2860} and using \eqref{theta-lower}
proves \eqref{2118}. The calculations \eqref{sixth-temperature}--\eqref{sixth-peak} now apply also for $\beta>1$.
\qed

\subsection*{2.2. Higher order estimates}
\begin{lemma}\label{lm280}
Suppose that
\begin{equation}\label{all-smallness}
 \begin{gathered}
 0\leq\alpha\leq M^{-1},\qquad
 \alpha C_1(1+M)^3\leq\min\{1,C_2/4\},\\
 \alpha\left(\frac{2eC_3C_4}{\tilde\mu}+\sqrt2 C_3^2C_4\right)
                                      (1+M)^2\leq1,\\
 \alpha^2(e\tilde\mu C_4)^2M\leq1.
 \end{gathered}
\end{equation}
There is $C_7>1$, independent of $M,T,\alpha$, such that
\begin{equation}\label{292}
 \begin{aligned}
 &\sup_{0\leq t\leq T}
       \left\{\|(v,v^{-1},\theta,\theta^{-1})(t)\|_{L^\infty}
                     +\|(v-1,u,\theta-1)(t)\|_{H^1}^2\right\}\\
 &\quad+\int_0^T
                  \|(v_x,u_x,\theta_x,u_t,\theta_t,u_{xx},\theta_{xx})\|_{L^2}^2dt
 \leq C_7.
 \end{aligned}
\end{equation}
Moreover, $(v-1,u,\theta-1)\in C([0,T];H^1)$.
In addition, $v_t=u_x$ and $v_{xt}=u_{xx}$ have the bounds in
\eqref{1110}.
\end{lemma}
\noindent{\it Proof.}
The preceding lemmas and the entropy estimate give
\begin{align}
 &\sup_{0\leq t\leq T}
       \left\{\|(v,v^{-1},\theta,\theta^{-1})(t)\|_{L^\infty}
                     +\|(v-1,u,\theta-1)(t)\|_{H^1}^2
                     +\|\theta^\beta\theta_x(t)\|_{L^2}^2\right\}\notag\\
 &\quad+\int_0^T
           \left(\|(v_x,u_x,\theta_x,u_t,\theta_t,u_{xx})\|_{L^2}^2
                            +\|\theta^\beta\theta_x\|_{L^2}^2\right)dt
 \leq C,\label{293}\\
 &\int_0^T\int_\Omega u_x^4dxdt
 \leq2\sup_t\|u_x\|_{L^2}^2
       \left(\int_0^T\|u_x\|_{L^2}^2dt\right)^{1/2}
       \left(\int_0^T\|u_{xx}\|_{L^2}^2dt\right)^{1/2}\leq C.
 \label{ux-fourth}
\end{align}
Expanding \eqref{13a},
\begin{equation}\label{flux-identity}
 \frac{(\theta^\beta\theta_x)_x}v
 =\frac{\theta^\beta\theta_xv_x}{v^2}
       +\frac{c_v}{\tilde\kappa}\theta_t
       -\frac{\mu}{\tilde\kappa v}u_x^2
       +\frac{R\theta}{\tilde\kappa v}u_x.
\end{equation}
Consequently,
\begin{align}
 \int_0^T\left\|\frac{(\theta^\beta\theta_x)_x}v\right\|_{L^2}^2dt
 &\leq C\sup_t\|v_x\|_{L^2}^2
                         \int_0^T\|\theta^\beta\theta_x\|_{L^\infty}^2dt
       +C\int_0^T\int_\Omega(\theta_t^2+u_x^4+u_x^2)dxdt\notag\\
 &\leq C\int_0^T\|\theta^\beta\theta_x\|_{L^\infty}^2dt+C,
 \label{294}\\
 \int_0^T\|\theta^\beta\theta_x\|_{L^\infty}^2dt
 &\leq2\int_0^T\|\theta^\beta\theta_x\|_{L^2}
                            \|(\theta^\beta\theta_x)_x\|_{L^2}dt\notag\\
 &\leq\eta\int_0^T
                   \left\|\frac{(\theta^\beta\theta_x)_x}v\right\|_{L^2}^2dt
       +C(\eta)\int_0^T\|\theta^\beta\theta_x\|_{L^2}^2dt.
 \label{295}
\end{align}
Before absorption, \eqref{20} and \eqref{finite-products} give
\begin{align*}
 \int_0^T\int_\Omega|(\theta^\beta\theta_x)_x|^2dxdt
 &\leq C(M)\int_0^T\int_\Omega
                 (\theta_{xx}^2+\theta_x^4)dxdt\leq C(M),\\
 \int_0^T\|\theta^\beta\theta_x\|_{L^\infty}^2dt
 &\leq C(M)\int_0^T\|\theta_x\|_{L^\infty}^2dt\leq C(M).
\end{align*}
Choose $\eta>0$ so that the first term on substituting
\eqref{295} into \eqref{294} has coefficient at most $1/2$. Then
\begin{align}
 \int_0^T\left\|\frac{(\theta^\beta\theta_x)_x}v\right\|_{L^2}^2dt
 &\leq\frac12\int_0^T
                   \left\|\frac{(\theta^\beta\theta_x)_x}v\right\|_{L^2}^2dt+C,
 \notag\\
 \int_0^T\left(\|(\theta^\beta\theta_x)_x\|_{L^2}^2
                         +\|\theta^\beta\theta_x\|_{L^\infty}^2\right)dt
 &\leq C.\label{flux-bound}
\end{align}
Finally,
\begin{align}
 \theta_{xx}
 &=\theta^{-\beta}(\theta^\beta\theta_x)_x
                                      -\beta\theta^{-1}\theta_x^2,\notag\\
 \int_0^T\int_\Omega\theta_x^4dxdt
 &\leq\sup_t\|\theta_x\|_{L^2}^2\int_0^T\|\theta_x\|_{L^\infty}^2dt
 \leq C\int_0^T\|\theta^\beta\theta_x\|_{L^\infty}^2dt\leq C,\notag\\
 \int_0^T\|\theta_{xx}\|_{L^2}^2dt
 &\leq C\int_0^T\|(\theta^\beta\theta_x)_x\|_{L^2}^2dt
                              +C\int_0^T\int_\Omega\theta_x^4dxdt
 \leq C.
 \label{thexx}
\end{align}
Together with \eqref{293}, this proves \eqref{292}. The remaining
assertions follow directly from $v_t=u_x$.

Finally, we verify the time continuity and the time-integrated identities.

Let $\Pi_n$ be the spectral projection onto $[0,n]$ of
$-\partial_{xx}$ with the prescribed homogeneous boundary condition.
For a function $z$ with this boundary condition and
\begin{equation*}
 \sup_{t\leq T}\int_\Omega z^2dx
 +\int_0^T\int_\Omega(z_t^2+z_x^2+z_{xx}^2)dxdt\leq C(T,M),
\end{equation*}
we have
\begin{align}
 &\frac d{dt}\|(\Pi_nz)_x\|_{L^2}^2
 =-2\int_\Omega(\Pi_nz)_{xx}\Pi_nz_tdx,\notag\\
 &\sup_{0\leq t\leq T}\|((\Pi_m-\Pi_n)z)_x(t)\|_{L^2}^2\notag\\
 &\qquad\leq\|((\Pi_m-\Pi_n)z)_x(s_0)\|_{L^2}^2\notag\\
 &\qquad\quad+\int_0^T\int_\Omega
       \bigl(|(\Pi_m-\Pi_n)z_{xx}|^2+|(\Pi_m-\Pi_n)z_t|^2\bigr)dxdt
 \longrightarrow0,\label{spectral-tail}
\end{align}
where $m>n\to\infty$ and $s_0\in(0,T)$ is chosen so that
$\int_\Omega(z^2+z_x^2+z_{xx}^2)(x,s_0)dx\leq C(T,M)$. Also,
\begin{equation*}
 \|z(t)-z(s)\|_{L^2}^2
 \leq(t-s)\int_s^t\int_\Omega z_t^2dxdr\qquad(0\leq s<t\leq T).
\end{equation*}
It follows that
\begin{align}
 &\sup_{t\leq T}\int_\Omega
       \bigl(|(\Pi_n-1)z|^2+|((\Pi_n-1)z)_x|^2\bigr)dx
 \longrightarrow0,\notag\\*
 &\|z_x(t)\|_{L^2}^2-\|z_x(s)\|_{L^2}^2
 =-2\int_s^t\int_\Omega z_{xx}z_t\,dxdr.
 \label{gradient-chain}
\end{align}
Apply this to $u$ and $\theta-1$. In particular,
\begin{equation}\label{velocity-chain}
 \frac12\frac d{dt}\|u_x\|_{L^2}^2=-\int_\Omega u_{xx}u_tdx.
\end{equation}
For $v$, use
\begin{equation}
 \|v(t)-v(s)\|_{H^1}^2
 \leq(t-s)\int_s^t\int_\Omega(u_x^2+u_{xx}^2)dxdr.
\end{equation}
On the half-line,
\begin{equation}\label{neumann-trace}
 \int_0^T|\theta_x(0,t)|^2dt
 \leq\int_0^T\int_\Omega(\theta_x^2+\theta_{xx}^2)dxdt\leq C(M).
\end{equation}
For the initial zero-order trace,
\begin{equation}
 \begin{aligned}
 |z(0,t)-z(0,0)|^2
 &\leq2\|z(t)-z(0)\|_{L^2}
                    \bigl(\|z_x(t)\|_{L^2}+\|z_x(0)\|_{L^2}\bigr)\\
 &\leq2\sqrt t\left(\int_0^t\|z_t\|_{L^2}^2ds\right)^{1/2}
                    \bigl(\|z_x(t)\|_{L^2}+\|z_x(0)\|_{L^2}\bigr)
 \longrightarrow0.
 \end{aligned}
\end{equation}

For the weighted identity, differentiation gives
\begin{align}
 \left(\frac{\theta^{\beta+1}-1}{\beta+1}\right)_t
 &=\theta^\beta\theta_t,\qquad
 \left(\frac{\theta^{\beta+1}-1}{\beta+1}\right)_{xx}
 =\theta^\beta\theta_{xx}+\beta\theta^{\beta-1}\theta_x^2,\notag\\
 (v^{-1})_x&=-\frac{v_x}{v^2},\qquad
 (v^{-1})_t=-\frac{u_x}{v^2},\qquad
 (v^{-1})_{tx}=-\frac{u_{xx}}{v^2}+\frac{2u_xv_x}{v^3}.\notag
\end{align}
By \eqref{20} and \eqref{finite-products},
\begin{align}
 &\sup_{t\leq T}\int_\Omega
       \left|\frac{\theta^{\beta+1}-1}{\beta+1}\right|^2dx
 +\int_0^T\int_\Omega
       \bigl(|\theta^\beta\theta_t|^2+|\theta^\beta\theta_x|^2
            +|\theta^\beta\theta_{xx}
                         +\beta\theta^{\beta-1}\theta_x^2|^2\bigr)dxdt
 \notag\\
 &\quad\leq C(M)\left[\sup_{t\leq T}\int_\Omega(\theta-1)^2dx
       +\int_0^T\int_\Omega
                 (\theta_t^2+\theta_x^2+\theta_{xx}^2+\theta_x^4)dxdt\right]
 \leq C(M),\notag\\
 &\sup_{t\leq T}\int_\Omega|(v^{-1})_x|^2dx
 +\int_0^T\|(v^{-1})_t\|_{L^\infty}^2dt
 +\int_0^T\int_\Omega|(v^{-1})_{tx}|^2dxdt\notag\\
 &\quad\leq C(M)\left[\sup_{t\leq T}\int_\Omega v_x^2dx
       +\int_0^T\|u_x\|_{L^\infty}^2dt
       +\int_0^T\int_\Omega(u_{xx}^2+u_x^2v_x^2)dxdt\right]
 \leq C(M).\label{weighted-finite}
\end{align}
For $z$ as above, \eqref{weighted-finite} gives
\begin{align}
 &\frac d{dt}\int_\Omega v^{-1}|(\Pi_nz)_x|^2dx
 =-2\int_\Omega\bigl(v^{-1}(\Pi_nz)_x\bigr)_x\Pi_nz_t\,dx
   +\int_\Omega (v^{-1})_t|(\Pi_nz)_x|^2dx,\notag\\
 &\int_0^T\|((\Pi_n-1)z)_x\|_{L^\infty}^2dt\notag\\
 &\qquad\leq\int_0^T\int_\Omega
       \bigl(|((\Pi_n-1)z)_x|^2+|((\Pi_n-1)z)_{xx}|^2\bigr)dxdt
 \longrightarrow0,\notag\\
 &\int_0^T\int_\Omega
       \left|\bigl(v^{-1}((\Pi_n-1)z)_x\bigr)_x\right|^2dxdt\notag\\*
 &\qquad\leq2M^2\int_0^T\int_\Omega|((\Pi_n-1)z)_{xx}|^2dxdt
       +2M^5\int_0^T\|((\Pi_n-1)z)_x\|_{L^\infty}^2dt
 \longrightarrow0,\notag\\
 &\int_0^T\int_\Omega|(v^{-1})_t|
       \bigl|| (\Pi_nz)_x|^2-|z_x|^2\bigr|\,dxdt\notag\\
 &\qquad\leq\sqrt T
       \left(\int_0^T\|(v^{-1})_t\|_{L^\infty}^2dt\right)^{1/2}
       \sup_t\|((\Pi_n-1)z)_x\|_{L^2}
       \sup_t(\|(\Pi_nz)_x\|_{L^2}+\|z_x\|_{L^2})\notag\\
 &\qquad\leq C(T,M)
       \left(\sup_t\int_\Omega|((\Pi_n-1)z)_x|^2dx\right)^{1/2}
 \longrightarrow0.\notag
\end{align}
Thus
\begin{equation}\label{weighted-chain}
 \left[\int_\Omega v^{-1}z_x^2dx\right]_s^t
 =-2\int_s^t\int_\Omega(v^{-1}z_x)_xz_t\,dxdr
   +\int_s^t\int_\Omega (v^{-1})_tz_x^2\,dxdr.
\end{equation}
Applying \eqref{weighted-chain} to $(\theta^{\beta+1}-1)/(\beta+1)$ gives
\begin{equation}\label{thermal-chain}
 \frac12\frac d{dt}\int_\Omega\frac{\theta^{2\beta}\theta_x^2}v dx
 =-\int_\Omega\left(\frac{\theta^\beta\theta_x}v\right)_x
                         \theta^\beta\theta_tdx
 -\frac12\int_\Omega\frac{u_x\theta^{2\beta}\theta_x^2}{v^2}dx.
\end{equation} The spectral truncations preserve the appropriate
homogeneous boundary condition.

The lower-order tests follow from the $H^1(0,T;L^2)$ chain rule,
\eqref{finite-products}, \eqref{cross-chain} and
\eqref{moving-level-chain}. In particular, positive-part approximation gives
\begin{equation}
 \frac d{dt}\int_\Omega u^2(\theta-2)_+dx
 =\int_\Omega\bigl(2uu_t(\theta-2)_+
              +u^2{\bf1}_{\{\theta>2\}}\theta_t\bigr)dx.
\end{equation}
The temperature multipliers vanish at a fixed-temperature boundary,
and the heat flux vanishes at an insulated boundary. The velocity
multipliers have zero trace. Spatial cutoffs remove products of $L^2$
factors at infinity. Thus all time-integrated identities include their
initial values, without an initial $H^2$ assumption.

\qed

\section{Proof of Theorem \ref{lm11}}\label{3}
\setcounter{equation}{0}

We first recall the local existence result.
\begin{lemma}\label{lm31}
Suppose that the initial data satisfy \eqref{190} and the zero-order
compatibility conditions in Theorem~\ref{lm11}. For fixed positive
$R,c_v,\tilde\mu,\tilde\kappa$ and fixed $\alpha,\beta\geq0$, there is
$T_0>0$ such that the corresponding problem admits a unique solution
in $X([0,T_0])$. The solution can be continued beyond a finite endpoint
provided that the perturbation $H^1$ norm remains bounded and $v,\theta$
remain bounded above and away from zero.
\end{lemma}
The local construction uses the continuity equation and the two
uniformly parabolic equations for $u$ and $\theta$. On a fixed positive
range of $v,\theta$, the coefficients and their first derivatives are
bounded, and the usual local energy estimates apply to the $H^1$
perturbations. The same estimates give uniqueness and continuation.
The boundary conditions are imposed at positive times; the initial
traces and the required time continuity are as in
Lemma~\ref{lm280}.

\subsection{Global existence}
If $E_0=0$, the assertion follows from the equilibrium solution.
For $E_0>0$, set
\begin{align}
 M&=8\max\left\{1,C_7,(\inf_\Omega v_0)^{-1},
                       (\inf_\Omega\theta_0)^{-1},\right.\notag\\
 &\hspace{23mm}\left.
       1+\sqrt2\|(v_0-1,u_0,\theta_0-1)\|_{H^1},
       \|(v_0-1,u_0,\theta_0-1)\|_{H^1}^2\right\},\label{fixed-M}\\
 \varepsilon_0
 &=\min\left\{M^{-1},
       \frac{\min\{1,C_2/4\}}{C_1(1+M)^3},
       \right.\notag\\
 &\hspace{18mm}\left.
       \frac1{\left(2eC_3C_4/\tilde\mu+\sqrt2 C_3^2C_4\right)(1+M)^2},
       \frac1{e\tilde\mu C_4\sqrt M}\right\}>0.
 \label{fixed-alpha}
\end{align}
The constants $C_1,C_2,C_3,C_4,C_7$ were obtained before choosing $M$, and
depend only on the fixed coefficients and initial quantities.
For any $0\leq\alpha\leq\varepsilon_0$, \eqref{all-smallness} holds.
Moreover,
\begin{equation}
 \|v_0-1\|_{L^\infty}^2\leq2\|v_0-1\|_{L^2}\|v_{0x}\|_{L^2},\qquad
 \|\theta_0-1\|_{L^\infty}^2\leq2\|\theta_0-1\|_{L^2}\|\theta_{0x}\|_{L^2}.
\end{equation}
The local solution therefore belongs to $X_M$ on a sufficiently
short interval. On each interval on which it belongs to $X_M$,
Lemma~\ref{lm280} gives the strict improvement
\begin{equation}\label{strict-improvement}
 \begin{gathered}
 8/M\leq v,\theta\leq M/8,\\
 \sup_{0\leq t\leq T}\|(v-1,u,\theta-1)(t)\|_{H^1}^2
 +\int_0^T\|(v_x,u_x,\theta_x,\theta_t,u_{xx},\theta_{xx})\|_{L^2}^2dt
 \leq M/8.
 \end{gathered}
\end{equation}
The $H^1$ continuity in Lemma~\ref{lm280} also gives continuity
in $L^\infty$. Moreover,
\begin{equation*}
 \int_t^{t+h}\int_\Omega
       (v_x^2+u_x^2+\theta_x^2+\theta_t^2+u_{xx}^2+\theta_{xx}^2)dxds
 \longrightarrow0\qquad(h\to0).
\end{equation*}
Thus a first finite exit from $X_M$ is impossible. At a finite maximal
existence time, \eqref{strict-improvement} also satisfies the continuation
criterion in Lemma~\ref{lm31}. Thus the solution is global.
Taking $T\uparrow\infty$ in \eqref{292} proves \eqref{1110} and
\eqref{18}. Uniqueness follows from Lemma~\ref{lm31}.

\subsection{Large time behavior}
On the global solution just obtained, \eqref{292} and
\eqref{flux-bound} hold with $T=\infty$. Applying
\eqref{gradient-chain} to $u$ and
$(\theta^{\beta+1}-1)/(\beta+1)$ gives
\begin{align}
 &\frac12\frac d{dt}\|u_x\|_{L^2}^2
 =-\int_\Omega u_{xx}u_tdx,\notag\\
 &\frac12\frac d{dt}\|\theta^\beta\theta_x\|_{L^2}^2
 =-\int_\Omega(\theta^\beta\theta_x)_x\theta^\beta\theta_tdx,
 \notag\\
 &\frac12\frac d{dt}\|v_x\|_{L^2}^2
 =\int_\Omega v_xu_{xx}dx,\notag\\
 &\int_0^\infty
 \left|\frac d{dt}
       \bigl(\|v_x\|_{L^2}^2+\|u_x\|_{L^2}^2+\|\theta^\beta\theta_x\|_{L^2}^2\bigr)\right|dt
 \notag\\
 &\qquad\leq2\int_0^\infty
       \bigl(\|v_x\|_{L^2}\|u_{xx}\|_{L^2}+\|u_t\|_{L^2}\|u_{xx}\|_{L^2}\bigr)dt\notag\\
 &\qquad\quad+2\int_0^\infty\|\theta^\beta\theta_t\|_{L^2}
                            \|(\theta^\beta\theta_x)_x\|_{L^2}dt
 \leq C.\label{gradient-variation}
\end{align}
It follows that
\begin{align}
 &\|v_x(t)\|_{L^2}^2+\|u_x(t)\|_{L^2}^2+\|\theta^\beta\theta_x(t)\|_{L^2}^2\notag\\
 &\quad\leq\int_t^{t+1}
       \bigl(\|v_x(s)\|_{L^2}^2+\|u_x(s)\|_{L^2}^2
                              +\|\theta^\beta\theta_x(s)\|_{L^2}^2\bigr)ds
 \notag\\
 &\qquad+\int_t^{t+1}
       \left|\frac d{ds}
        \bigl(\|v_x\|_{L^2}^2+\|u_x\|_{L^2}^2+\|\theta^\beta\theta_x\|_{L^2}^2\bigr)\right|ds
 \longrightarrow0,\notag\\
 &\|(v_x,u_x,\theta_x)(t)\|_{L^2}\longrightarrow0.
 \label{gradient-decay}
\end{align}
Likewise, \eqref{sixth-temperature} gives
\begin{equation}
 \begin{aligned}
 \int_\Omega(\theta-1)^6(x,t)dx
 &\leq\int_t^{t+1}\int_\Omega(\theta-1)^6(x,s)dxds\\
 &\quad+\int_t^{t+1}
          \left|\frac d{ds}\int_\Omega(\theta-1)^6(x,s)dx\right|ds
 \longrightarrow0.
 \end{aligned}
\end{equation}
For $f=v-1,u,\theta-1$ and $2<p<\infty$,
\begin{align}
 \|f(t)\|_p
 &\leq\|f(t)\|_{L^2}^{2/p}\|f(t)\|_{L^\infty}^{1-2/p}\notag\\
 &\leq2^{(p-2)/(2p)}
       \|f(t)\|_{L^2}^{(p+2)/(2p)}
       \|f_x(t)\|_{L^2}^{(p-2)/(2p)}
 \longrightarrow0,\notag\\
 \|f(t)\|_{L^\infty}^2
 &\leq2\|f(t)\|_{L^2}\|f_x(t)\|_{L^2}\longrightarrow0.
\end{align}
This proves \eqref{110} and completes the proof of
Theorem~\ref{lm11}.

\end{document}